\documentclass{amsart}
\usepackage{amsmath}
\usepackage{amsthm}
\usepackage{amssymb}
\usepackage{color}
\usepackage{graphicx}
\usepackage{apptools}
\usepackage{chngcntr}
\usepackage[toc,page]{appendix}

\newtheorem{thm}{Theorem}\newtheorem{lemma}[thm]{Lemma}

\newtheorem{defi}[thm]{Definition}
\newtheorem{prop}[thm]{Proposition}
\newtheorem{rk}[thm]{Remark}

\AtAppendix{\counterwithin{thm}{section}}
\AtAppendix{\counterwithin{equation}{section}}

\newcommand{\rr}{{\mathbb{R}}}

\newcommand{\nn}{{\mathbb{N}}}

\newcommand{\E}{{\mathbb{E}}}

\newcommand{\PP}{{\mathbb{P}}}
\newcommand{\QQ}{{\mathbb{Q}}}
\newcommand{\e}{\varepsilon}
\newcommand{\vip}{\vskip.2cm}
\newcommand{\indiq}{\hbox{\rm 1}{\hskip -2.8 pt}\hbox{\rm I}}
\newcommand{\dd}{{\rm d}}

\newcommand{\cS}{{\mathcal{S}}}

\newcommand{\cL}{{\mathcal{L}}}
\newcommand{\cK}{{\mathcal{K}}}

\newcommand{\cF}{{\mathcal{F}}}
\newcommand{\cP}{{\mathcal{P}}}
\newcommand{\cA}{{\mathcal{A}}}
\newcommand{\cB}{{\mathcal{B}}}

\newcommand{\cE}{{\mathcal{E}}}

\newcommand{\ig}{[\![}
\newcommand{\id}{]\!]}

\begin{document}

\title[Keller-Segel equation]
{ Weak convergence of the empirical measure for the Keller-Segel model in both subcritical and critical cases.}

\author{Yoan Tardy}
\address{Sorbonne Universit\'e, LPSM-UMR 8001, Case courrier 158, 75252 Paris Cedex 05, France.}
\email{nicolas.fournier@sorbonne-universite.fr, yoan.tardy@sorbonne-universite.fr}

\begin{abstract}
We show the weak convergence, up to extraction of a subsequence, of the empirical measure for the Keller-Segel system of particles in both subcritical and critical cases,
for general initial conditions. This particle system consists of $N$
planar Brownian motions interacting through a Coulombian attractive force, which is quite singular.
In the subcritical case, a stronger result has been established by
Bresch-Jabin-Wang \cite{bjw} at the price of two simplifications: the whole space $\rr^2$ is
replaced by a torus and the initial condition is assumed to be regular. 
In the subcritical case, our proof is fairly straightforward: we use
a {\it two particles} moment argument, which shows that particles do not aggregate in finite time, 
uniformly in the number of particles. The critical case requires more work.
\end{abstract} 
\maketitle
\section{Introduction and results}

\subsection{The model}\label{ffs}
We consider the classical Keller-Segel model of chemotaxis (or Patlak-Keller-Segel model) in the parabolic-elliptic case, which writes 
\begin{align}\label{EDP}
\partial_t f =& \frac12 \Delta f - 2\pi \theta \nabla \cdot ( f \nabla c ), \\
0 =& \Delta c + f , \label{EDP2}
\end{align}
 with initial condition $f_0\in \cP (\rr^2)$, where $\cP (\rr^2)$ denotes the set of probability measures on $\rr^2$, where $(f,c)$ is the unknown of the equation and $f_t(x)$ and $c_t(x)$ respectively stand for the density of particles and chemoattractant at position $x\in \rr^2$ and time $t\geq 0$. The scalar $\theta>0$ describes the intensity of the attraction of the chemoattractant. 
 \begin{rk}
In the literature, \eqref{EDP} is written 
\begin{align*}
\partial_t f =&  \Delta f - \chi\nabla \cdot ( f \nabla c ),\\
0 =& \Delta c + f,
\end{align*} 
and $f_0$ is assumed to have a finite mass $M>0$. One can easily convince oneself that setting $\tilde f_t(x) =  f_{t/2}(x)/M$ and $ c_{t/2}(x)/M=\tilde c_t(x) $, $(\tilde f, \tilde c )$ solves \eqref{EDP}-\eqref{EDP2} with $\theta = \chi M / (4\pi)$. Observe that if $\chi =1$, one recovers that the classical threshold $M=8\pi$ of \eqref{EDP}-\eqref{EDP2} corresponds to $\theta =2$ in our case. 
\end{rk}
Since $c$ is the solution of the Poisson equation on $\rr^2$ with $f$ as source term, the equation \eqref{EDP} can we rewritten 
\begin{align}\label{MVEDP}
\partial_t f =& \frac 12 \Delta f - \theta\nabla \cdot ( f (K\star f) ),
\end{align}
with
$$
 K(x) = -x/\|x\|^2 \quad  \mbox{ if } x\ne 0\quad \mbox{ and }\quad K(0)=0,
$$
 and $\star$ stands for the convolution product. One can understand this equation as the modeling of an infinite number of particles moving as Brownian motions which attract each other with a Coulombian force.

\vip

The main interest of this equation is the tight competition between the diffusion and the concentration of particles. As Blanchet-Dolbeaut-Perthame have shown in \cite{bdp}, introducing the variance of the problem $V_t = \int_{\rr^2} \|x-\int_{\rr^2} x f_t(\dd x)\|^2 \dd f_t(x)$, we informally have, using several integrations by parts, that $\frac{\dd }{\dd t }V_t = (2-\theta).$ This highlights that there are informally three cases:
\begin{itemize}
\item  the subcritical case, which corresponds to the case where $\theta < 2$: $V_t$ tends to infinity as time goes to infinity which suggests that the diffusion is dominant over the concentration,
\item the critical case where $\theta =2$,
\item the supercritical case, where $\theta >2$: $V_t$ ends up being non-positive, which suggests that a blow-up occurs.
\end{itemize}

\subsection{Weak solutions}

In this paper, we will deal with a very weak type of solution that we define here, 
see e.g. the paper of Blanchet-Dolbeault-Perthame  \cite{bdp}.
\begin{defi}
We say that $f\in C([0,\infty),\cP(\rr^2))$ is a weak solution of \eqref{EDP}-\eqref{EDP2} with initial condition $f_0\in \cP(\rr^2)$ if for all $\varphi \in C^2_b(\rr^2)$, 
\begin{align*}
\int_{\rr^2} \varphi (x) f_t(\dd x) =& \int_{\rr^2} \varphi (x) f_0(\dd x) + \frac12 \int_0^t\int_{\rr^2} \Delta \varphi (x) f_s(\dd x)\dd s \\
&+ \frac\theta 2 \int_0^t \int_{\rr^2} \int_{\rr^2}K(x-y)\cdot [\nabla \varphi(x) - \nabla \varphi (y) ] f_s(\dd x) f_s(\dd y) \dd s.
\end{align*}
\end{defi}
Observe that since $K(0) = 0$ and $\varphi \in C^2_b$, the last integral is well defined. Since $(x,y)\mapsto K(x-y)\cdot [\nabla \varphi (x) - \nabla \varphi (y) ]$ equals $0$ ( arbitrarily) on the set $\{ (x,y) \in \rr^2 : x=y\}$, we expect that a satisfactory solution must verify $\int_0^t  \indiq_{\{x= y\}} f_s(\dd x) f_s(\dd y) \dd s = 0$. 

\subsection{The associated trajectories}

The equation \eqref{MVEDP} comes from an Eulerian point of view. The following SDE models the same problem with a Lagrangian point of view: we consider the behavior of one typical particle among the other.
\begin{defi}
We say that $(X_t)_{t\geq 0}$ is a solution of the stochastic Keller-Segel equation with initial condition $X_0$ if
\begin{align}\label{EDSMV}
\mbox{for all } t\geq 0, \quad X_t = X_0 + B_t + \theta \int_0^t K\star f_s (X_s)\dd s,
\end{align}
where $(B_t)_{t\ge 0}$ is a $2$-dimensional Brownian motion and $f_s$ is the law of $X_s$. 
\end{defi}
The idea is to follow the motion of one specific particle instead of considering the proportion of 
particles at a precise place and a precise time. One easily checks from the It\^o formula that
if $(X_t)_{t\geq 0}$ is a solution of the stochastic Keller-Segel equation \eqref{EDSMV}, then $(f_t)_{t\geq 0}$ is
a weak solution to \eqref{EDP}-\eqref{EDP2}.

\vip

We consider for $N\ge 2$ and $\theta>0$ a natural discretization introduced by Keller-Segel in \cite{ks2} of the stochastic Keller-Segel equation  \eqref{EDSMV}: 
we consider $N\ge 2$ particles with positions $(X^{i,N}_t)_{i\in \ig 1,N\id}$, solving (recall that $K(0)=0$)
\begin{equation}\label{EDS}
 X^{i,N}_t = X^{i,N}_0 + B^i_t + \frac{\theta}{N}\int_0^t\sum_{j=1 }^N K(X^{i,N}_s-X^{j,N}_s)\dd s,
\end{equation}
where $((B^i_t)_{i\in \ig 1,N\id})_{t\ge 0}$ is a family of $N$ independent $2$-dimensional Brownian motion independent from $(X^{i,N}_0)_{i\in \ig 1,N \id }$. We call such a process a $KS(\theta,N)$-process,  $\theta \le 2(N-2)/(N-1)$ and its existence is guaranteed by \cite[Theorem 7]{fj} for any initial datum $(X_0^{i,N})_{i\in \ig 1,N \id}$ with its law $F^N_0$ in $\cP_{sym,1}^*((\rr^2)^N)$, where $\cP_{sym,1}^*((\rr^2)^N)$ denotes the set of exchangeable probability measures on $(\rr^2)^N$ with a finite 
$1$-order moment and such that
$$
F_0^N(\{ \mbox{There exists } i\ne j\in \ig 1,N\id \mbox{ such that } x^i = x^j \})=0.
$$
Moreover, $(X^{i,N}_t)_{i\in \ig 1,N\id}$ is exchangeable and we have, for all $t>0$,
$$
\int_0^t ||K(X^{1,N}_s-X^{2,N}_s)||\dd s<\infty.
$$ 
\subsection{The subcritical case}
We endow the space $C([0,\infty),\cP(\rr^2))$ with the topology of the uniform convergence on the compact sets, $\cP (\rr^2)$ being endowed with the topology of the weak convergence. Our goal is to prove the following result.
\begin{thm}\label{thm1}
Let $\theta \in (0,2)$ and $f_0 \in \cP(\rr^2)$. For each $N\ge N_0 := (1+ \lceil 2/(2-\theta) \rceil)\vee 5$, consider $F_0^N \in \cP_{sym,1}^* ((\rr^2)^N)$ and a $KS (\theta,N)$-process $(X^{i,N}_t)_{t \ge 0, i\in \ig 1,N\id}$ with initial law $F^N_0$, as well as the empirical measure for all $t\ge 0$, $\mu^{N}_t := N^{-1} \sum_{i=1}^N \delta _{X^{i,N}_{t}}$, which of course a.s. belongs to $\cP (\rr^2)$. We assume that $\mu^N_0$ goes weakly to $f_0 $ in probability as $N\to \infty$.

\vip

(i) The sequence $( (\mu^N_t)_{t\ge 0})_{N\ge N_0}$ is tight in $C([0,\infty), \cP (\rr^2))$.

\vip

(ii) For any sequence $(N_k)_{k\ge 0}$ such that $(\mu^{N_k}_t)_{t\ge 0}$ goes in law in $C([0,\infty), \cP(\rr^2))$ as $k\to \infty$ to some $(\mu_t)_{t\ge 0}$, this limit $(\mu_t)_{t\ge 0}$ is a.s. a weak solution to \eqref{EDP}-\eqref{EDP2} starting from $\mu_0 = f_0$. Moreover, for all $T>0$, all $\gamma \in (\theta,2)$.
$$
\E \Big[ \int_0^T \int_{\rr^2}\int_{\rr^2} \|x-y\|^{\gamma-2}\mu_t (\dd x) \mu_t(\dd y) \dd t \Big]< \infty.
$$
\end{thm}

Observe that the condition $N\ge N_0$ implies that the particle system exists thanks to \cite[Theorem 7]{fj}. There is no condition at all on $f_0$: it can be a Dirac mass, it does not need to have finite moments.

\vip

Let us mention that we impose the condition $F_0^N \in  \cP_{sym,1}^* ((\rr^2)^N)$ only to apply
\cite[Theorem 7]{fj}, which concerns the existence of the particle system. The condition is not used anywhere else
in the present paper.
When $\theta < 1/2$, a stronger existence result is established, see \cite[Theorem 5]{fj},
and we can assume only that $F_0^N \in  \cP_{sym,1} ((\rr^2)^N)$.

\vip
For any given $f_0 \in \cP(\rr^2)$ (possibly with some atomes and/or without a moment of order $1$),  
we can easily build $F_0^N$ satisfying all the assumptions: take for example
$F_0^N$ the law of 
$(\chi_N(X^i_0) + (1/N)G^i)_{i\in \ig 1,N\id}$ where $(X^i_0)_{i\ge 1}$ are i.i.d $f_0$-distributed random variables,
where $(G^i)_{i\ge 1}$ are i.i.d standard
bi-dimensional Gaussian random variables
and where $\chi_N(x_1,x_2)=(x_1\land N \lor(-N),x_2\land N \lor(-N))$.

\vip

We highlight that Theorem \ref{thm1} has to be linked with \cite{ft}, where the critical case
was not treated for the sake of conciseness. The relation between these results will be discussed in Section \ref{ref}

\subsection{Existence of the particle system in the critical or supercritical case}\label{criandsupercri}

In the case where $\theta>2(N-2)/(N-1)$, the existence of a process solving \eqref{EDS} {\it for all time} 
is not guaranteed. So we need to introduce a notion of local solution. 
If $\tau$ is  a stopping 
time for a 
filtration $(\cF_t)_{t\ge 0}$, a continuous $(\rr^2)^N$-valued process 
$(X^{i,N}_t)_{ t\in [0,\tau), i\in \ig 1,N \id}$ is a $KS(\theta,N)$-process on $[0,\tau)$ if for all sequences of increasing stopping times $(\tau_n)_{n\in \nn}$
with respect to the filtration $(\cF_t)_{t\ge 0}$ such that $\lim_{n\to \infty} \tau_n = \tau$ a.s. 
and for all $n\geq 1$, 
$(X^{i,N}_{t\land \tau_n})_{ t\ge 0, i\in \ig 1,N \id}$ is $(\cF_{t})_{t\geq 0}$-adapted and 
\eqref{EDS} holds true on $[0,\tau_n]$.

\vip

For the whole paper, we define for all $N\ge 2$, all $\ell\ge 1$, all $k\ge 2$, 
\begin{gather*}
\tau_k^{N,\ell }:= \inf \Big\{t>0 : \exists K\subset \ig 1,N\id : |K|=k \hbox{ and } R_K(X^N_t) \le \frac{1}{\ell} \Big\} \quad\mbox{ and }\quad\tau^N_k = \sup_{\ell\ge 1} \tau_k^{N,\ell},
\end{gather*}
where for all $K\subset \ig 1,N \id$, all $x =(x^1,\cdots , x^N) \in (\rr^2)^N$, 
$$
S_K(x) = \frac{1}{|K|}\sum_{i\in K} x^i  \qquad \mbox{ and } \qquad  R_K(x) = \sum_{i\in K} \|x^i-S_K(x)\|^2.
$$
The stopping time $\tau^{N,\ell}_k$ represents the first time when at least $k$ particles are $\ell^{-1}$-close.
Roughly, $\tau^N_k$ represents the first time when at least $k$ particles collide. However, this does not follow
from the definition, since the $KS(\theta,N)$-process could be discontinuous at time $\tau^{N}_k$.

\vip

One can find the following existence theorem in \cite[Theorem 7 and Lemma 14]{fj},
recall that our $\theta$ corresponds to $\chi/(4\pi)$ in \cite{fj}.
\begin{thm}[Fournier-Jourdain \cite{fj}]\label{fj1}
If $N\ge 5$, $\theta >0$ and $F^N_0 \in \cP_{sym,1}^*(\rr^2)$,
then there exists a $KS(\theta,N)$-process $(X_t^{i,N})_{t\in [0,\tau_{3}^N), i\in \ig 1,N\id}$ to \eqref{EDS}
with initial law $F^N_0$. The family 
$( (X^{i,N}_t)_{t\in [0,\tau_{3}^N)}, i\in \ig 1,N\id)$ is exchangeable and for any $\ell \ge 1$, any $t\geq 0$, we have
$$
\E \Big( \int_0^{t\wedge \tau_{3}^{N,\ell}} \|X^{1,N}_s-X^{2,N}_s\|^{-1}\dd s\Big) < \infty.
$$
Moreover $\tau_3^N < \infty$ a.s. if $\theta>2(N-2)/(N-1)$, and $\tau_3^N=\infty$ a.s. if
$\theta \in (0,2(N-2)/(N-1)]$.
\end{thm}
The local solution to \eqref{EDS} of \cite[Theorem 7]{fj} is built from \cite[Lemma 14]{fj}.
Hence we actually have $\E [\int_0^{t\wedge \tau_{3}^{N,\ell}} \|X^{1,N}_s-X^{2,N}_s\|^{\alpha-2}\dd s] < \infty,$ 
with some $\alpha <1$. This is stronger than the estimate we recalled in Theorem \ref{fj1}.

\vip

The existence of a global solution to \eqref{EDS} is false when $\theta \geq 2$. Indeed, consider for example 
in the critical case $\theta=2$. Assume that a global solution $(X^N_t)_{t\geq 0}$ exists. 
Then the process $R_{\ig 1,N\id}(X^N)$ would be a squared Bessel process with dimension $0$, 
see Lemma \ref{bessel}. Such a squared Bessel 
process reaches $0$ in finite time and then remains at $0$ for all times, meaning that all the particles
share the same location. As a consequence, the drift term of \eqref{EDS} cannot be well-defined.

\vip

However, the existence of the particle system until $\tau_3^N$ is sufficient to our purpose: 
in the critical case $\theta=2$, which is the object of our next main result,
it holds that $\tau_3^N$ tends to infinity in probability as $N\to \infty$. This will be shown in 
Section \ref{building}.

\vip
This is coherent with the fact that
the solution of the deterministic Keller-Segel equation \eqref{EDP}--\eqref{EDP2} is global, see \cite{ft}.

\subsection{The critical case}
Our second main result is the following.
\begin{thm}\label{thm2}
Assume $\theta=2$. Let $f_0 \in \cP(\rr^2)$ that  checks $\max_{x\in \rr^2}f_0(\{x\})<1$. For each $N\ge N_0 :=5$, consider $F_0^N \in \cP_{sym,1}^* ((\rr^2)^N)$ and $(X^{i,N}_t)_{t\in [0,\tau_3^N), i\in \ig 1,N\id}$ a $KS (2,N)$-process on $[0,\tau_3^N)$ with initial law $F^N_0$. We set $\mu^{N}_t := N^{-1} \sum_{i=1}^N \delta _{X^{i,N}_{t}}$ for $t\in [0,\tau^N_3)$ and we assume that
$\mu^N_0$ goes weakly to $f_0$ in probability as $N\to \infty$.
\vip
There exists an increasing sequence of integers $(\ell_N)_{N\ge N_0}$ such that, setting
$\beta_N=\tau_3^{N,\ell_N}$, it holds that $\lim_{N\to \infty} \beta_N=\infty$ in probability and
\vip

(i) the sequence $( (\mu^{N}_{t\land \beta_N})_{t\ge 0})_{N\ge 5}$ is tight in $C([0,\infty), \cP (\rr^2))$;

\vip

(ii) for any sequence $(N_k)_{k\ge 0}$ such that $(\mu^{N_k}_{t\land \beta_{N_k}})_{t\ge 0}$ goes in law, 
as $k\to \infty$, in $C([0,\infty), \cP(\rr^2))$, to some $(\mu_t)_{t\ge 0}$, it holds that
$\mu_t$ is a.s a weak solution to \eqref{EDP} starting from $\mu_0 = f_0$. Moreover,
\begin{equation}\label{ggg}
\E \Big[ \int_0^\infty \int_{\rr^2}\int_{\rr^2} \indiq_{\{x=y\}}\mu_t (\dd x) \mu_t(\dd y) \dd t \Big]=0.
\end{equation}
\end{thm}
The only condition on $f_0$ is that it is not a full Dirac mass, it can have any atoms as it wants
provided none of these atoms have mass $1$.
\vip
Since $\beta_N\to \infty$, the above theorem implies that if we fix $T>0$,  
$(\mu^N_t)_{t\in [0,T]}$ is well-defined with high probability (as $N\to\infty$), is tight in some sense,
and the accumulation points of this sequence are almost surely weak solutions of \eqref{EDP}-\eqref{EDP2} 
on $[0,T]$.

\subsection{Comparison with available results and comments}\label{ref}

The Keller-Segel equation \eqref{EDP}-\eqref{EDP2} was first introduced by Patlak in \cite{p} and Keller and Segel in \cite{ks} as a model for chemotaxis. One can find exhaustive content on what is understood about this model in the recent book of Biler \cite{b} and the review paper of Arumugan-Tyagi \cite{at}.

\vip

The notion of propagation of chaos was introduced by  Kac in \cite{k}, the idea was to make a step toward the rigorous justification of the Boltzmann equation. The idea is to approximate the solution of a mesoscopic PDE by the solution to an associated microscopic problem, which involves interacting particles. The case of Lipschitz continuous interaction kernels is now well-understood, we mention for example McKean in \cite{mk}, Sznitman in \cite{sz}, M\' el\' eard in  \cite{mel} and Mischler-Mouhot in \cite{mm} for significant contributions to the theory. The case of singular kernels is much more complicated and has been studied for different kinds of kernels concerning the viscous Burgers equation by Bossy-Talay in \cite{bt} and Jourdain in \cite{j}, the Dyson model by C\' epa-L\' epingle in \cite{cl}, etc.

\vip

A model relatively close to our subject is the $2d$-Navier-Stokes particle system studied by Osada in \cite{o} where the kernel $K$ is replaced by $ x^\perp/\|x\|^2$, which is as singular as the kernel of Keller-Segel, 
but is rotative instead of being attractive. This implies that collisions do not occur so that the singularity is not visited.
Osada obtained in \cite{o} the convergence of the corresponding stochastic
particle system under a large viscosity assumption.
This condition has been removed by Fournier-Hauray-Mischler in \cite{fhm} thanks to a compactness-uniqueness method, 
and then, more recently, Jabin-Wang obtained in\cite{jw} a quantitative convergence result,
using a {\it modulated free entropy} method. We mention Serfaty who introduced in \cite{serf} the modulated energy method in the deterministic case.  

\vip
Let us mention Stevens, who studied in \cite{stev}  the case where the chemoattractant is not in the stationary 
regime (parabolic-parabolic Keller-Segel equation), by considering a particle system
with two types of particles, representing bacteria and chemoattractant.

\vip

This paper uses some ideas already present in the papers of Fournier-Jourdain \cite{fj} and 
Fournier-Tardy \cite{ft}. In \cite{fj}, there is already a proof of Theorem 
\ref{thm1} in the case where $\theta<1/2$: the estimates there are not sufficiently precise to cover
the whole subcritical case $\theta<2$.
In \cite{ft}, there is a simple proof
of non-explosion of the weak solution to \eqref{EDP}-\eqref{EDP2} in the subcritical case. The proof relies
on a {\it two-particles} moment method which extends the
simple non-explosion proof elaborated in the paper of Biler-Karch-Lauren\c cot-Nadzieja \cite{bklnD, bklnP} 
concerning  the radially symmetric case, in which case the two-particles moment is not necessary since 
it actually rewrites as a simple one-particle moment. 
The method introduced in \cite{ft} is simple enough so that we are able to adapt it to show that
the particle system does not explode, uniformly in $N$ in some sense. Moreover, at the price of much more work,
we are able to include 
the critical case. We highlight that the result of \cite{ft} only concerns the deterministic 
equations \eqref{EDP}--\eqref{EDP2} and not the stochastic particle system \eqref{EDS}. 
The novelty of our result is about the convergence of the stochastic particle system to the weak global 
solution.

\vip

The main drawback of our results lies in the fact that we are not able to show the uniqueness of 
the weak solution of \eqref{EDP}--\eqref{EDP2} in order to conclude that the empirical measure truely 
converges weakly:
we only have the convergence of some subsequences. Let us mention that the uniqueness 
of mild solutions of the Keller-Segel equation is well-established (and easy) under a few regularity 
assumptions on the initial condition, 
see for example the paper of Olivera-Richard-Tomasevic \cite{ort}. However, 
it seems very difficult to prove that our {\it weak} solution
is {\it mild}.

\vip

Notice that we have not the least assumption on the limiting initial condition $f_0 \in \cP(\rr^2)$,
except in the critical case $\theta=2$ where we impose that $f_0$ is not a Dirac mass.
This last condition is necessary, since when $\theta=2$, Dirac masses are in some sense stationary solutions,
because we informally have 
$\frac \dd {\dd t} \int_{\rr^2} \|x-\int_{\rr^2} x f_t(\dd x)\|^2 \dd f_t(x)=2-\theta=0$, 
as mentioned in Subsection \ref{ffs}. A small drawback of our result is the condition
$F_0^N \in \cP^*_{sym,1}((\rr^2)^N)$. As already mentioned, we use nowhere this condition except to apply
the existence result (for the particle system) of Fournier-Jourdain \cite{fj}. We believe 
that when $\theta \in (0,2)$, there is actually existence of the particle system (at least for $N$ sufficiently
large) without this assumption, but this would require considerably much more work.

\vip

Concerning the mean-field limit of the Keller-Segel particle system, let us recall the main available results.

\vip

(a) Godinho-Quininao proved in \cite{gq} the convergence of the particle system, following the ideas 
of Fournier-Hauray-Mischler found in \cite{fhm}, when $K$ is replaced by $-x/(2\pi\|x\|^{1+\alpha})$ with $\alpha \in (0,1)$. This kernel is strictly less singular, this prevents any explosion issue and 
so the result is obtained without any discussion on some parameter such as $\theta$ in our case.

\vip

(b)  Olivera-Richard-Tomasevic have shown in \cite{ort} the convergence of a smoothed particle system  for every $\theta >0$ (even supercritical), where $K$ is replaced, roughly, by $-x/(\e_N+||x||^2)$ with $\e_N$ going to $0$ as
$N\to \infty$ in such a way that $\lim_N \sqrt N \e_N = \infty$. We highlight that the convergence is uniform on every compact set of $[0,T^*)$ where $T^*$ is the maximal time of the unique mild solution of the Keller-Segel equation (which is infinite in the subcritical case).  
The method is based on a semigroup approach developed by Flandoli \cite{fl}.

\vip

(c)   In Theorem $6$ of the paper of Fournier-Jourdain \cite{fj}, the convergence of the empirical measure is 
shown, up to extraction, when $\theta<1/2$, thus proving Theorem \ref{thm1} when $\theta <1/2$.
They have a light additional moment condition on $f_0$, as well as a (uniform in $N$) light moment condition on
$F_0^N$. However, they do not need to impose that $F_0^N \in \cP^*_{sym,1}((\rr^2)^N)$. When 
$\theta \in[1/2,2)$, Fournier-Jourdain have to assume that $F_0^N \in \cP^*_{sym,1}((\rr^2)^N)$ to prove 
the existence of the particle system when $\theta \in [1/2,2)$ (they do not study the limit
$N\to \infty$ in this case). We impose that $F_0^N \in \cP^*_{sym,1}((\rr^2)^N)$ precisely to use this existence
result.
Finally, it is worth emphasizing that the convergence in law we prove holds in $\cP (C(\rr_+, \rr^2))$, 
which is weaker than the convergence in law in $C(\rr_+, \cP(\rr^2))$ obtained in \cite{fj}. Again, this is
due to the fact that the more $\theta$ is large, the more the attraction is strong and the more the singularity 
of the kernel $K$ is visited.
With our proof, it might be possible to show a convergence in $\cP (C(\rr_+, \rr^2))$ when 
$\theta \in (0,1)$, but not when $\theta \in (1,2]$: roughly, our proof relies on the estimate
$\sup_{N\geq 1} \E[\int_0^t ||X^{1,N}_s-X^{2,N}_s||^{\gamma-2}]\dd s$ for all $\gamma \in (\theta,2)$, all $t\geq 0$, 
see Subsection \ref{map}. Recalling that $||K(x)||=||x||^{-1}$, this estimate
guarantees that the drift in \eqref{EDS} is well-controled, uniformly in $N\geq2$,
only when $\theta \in (0,1)$, in which case can choose $\gamma \in (\theta,2)$ such that $\gamma-2=-1$.

\vip

(d)  Finally,  Bresch-Jabin-Wang studied in \cite{bjw}  
the convergence of the empirical measure with quantitative estimates in the case where $\theta<2$ 
by using a {\it modulated free entropy } method in the same spirit as \cite{jw}.
%They do not really deal with \eqref{EDP}-\eqref{EDP2} but rather with a weak 
%solution to the associated Fokker-Planck P.D.E. 
They obtain a full convergence result (not only along a subsequence). However, they make two important
simplifications: they replace $\rr^2$ by a torus, and deal with regular initial conditions, 
namely at least $f_0$ of class $W^{2,\infty}$ on the torus.  Replacing $\rr^2$ by a torus
seems benign  from a heuristic point of view, but they widely use it during their proof.
In particular, this implies that the solution $f_t$ to the Keller-Segel equation is bounded from below,
which seems crucial to use some relative entropy technics.
\vip
To summarize, \cite{fj} and \cite{bjw} are the only available works dealing with a true attractive 
interaction in $1/r$, without any regularization. The first one is restricted to the very subcritical 
case $\theta \in (0,1/2)$. The second one treats the whole subcritical case $\theta \in (0,2)$,
but with $\rr^2$ replaced by a torus and with smooth initial conditions. Moreover, we believe
that in the subcritical case, our arguments are simpler than those of \cite{bjw} and we have
very few assumptions on the initial condition.
But of course, a drawback is that we prove convergence only up to
extraction of a subsequence. At last, it seems that nothing was known about the critical case
$\theta=2$.

\vip

Finally, let us mention that Cattiaux-Ped\` eches give in \cite{cp} another proof of the existence 
of a solution to \eqref{EDS}, relying on the use of Dirichlet forms, and they establish some 
weak uniqueness results.

\subsection{Main a priori estimate in the subcritical case}\label{map}

The subcritical case (which means the case where $\theta \in (0,2)$) is the easiest to handle and here we briefly present the main computation. The idea is to find an a priori estimate which testifies to the fact that particles don't have the tendency to aggregate. We fix $\gamma \in (\theta,2)$. We set 
$$
S_t = \E[\|X^{1,N}_t-X^{2,N}_t\|^{\gamma}\wedge 1 ].
$$
The idea is that when differentiating $S_t$, the diffusion makes appear something like
$\E[\|X^{1,N}_t-X^{2,N}_t\|^{\gamma -2}]$. If it dominates the drift term (which holds true in
the subcritical case for some well-chosen value of $\gamma$), this will imply, 
since $S_t$ is obviously bounded, an estimate on $\int_0^t\E[\|X^{1,N}_s-X^{2,N}_s\|^{\gamma -2}] \dd s$.
In some sense, this tells us that particles do not collide too often. 

\vip

 Informally using the It\^ o formula we get
 \begin{align}\label{rr}
 \| X^{1,N}_t - X^{2,N}_t\|^\gamma \wedge 1 = \| X^{1,N}_0 - X^{2,N}_0\|^\gamma \wedge 1 + M_t + \int_0^t  \indiq_{\|X^{1,N}_s-X^{2,N}_s\|\le 1}\Big( I^1_s + \frac{\theta}{N}I^2_s \Big)\dd s,
 \end{align}
 where $(M_t)_{t\ge 0}$ is a martingale, and where 
\begin{gather*}
 I^1_s = \Big(\gamma^2-\frac{2\gamma\theta}{N} \Big) \|X^{1,N}_s-X^{2,N}_s\|^{\gamma-2},   \\
I^2_s = \gamma\sum_{k=3}^N (K(X^{1,N}_s-X^{k,N}_s) + K( X^{k,N}_s-X^{2,N}_s)) \cdot (X^{1,N}_s-X^{2,N}_s)\|X^{1,N}_s - X^{2,N}_s\|^{\gamma-2}.
\end{gather*}
 
Taking expectations in \eqref{rr}, we find
\begin{align}\label{rrr}
S_t = S_0 + \int_0^t \E\Big[ \indiq_{\|X^{1,N}_s-X^{2,N}_s\|\le 1}\Big( I^1_s + \frac{\theta}{N}I^2_s \Big)\Big]\dd s.
\end{align}

Using exchangeability we get that $\E[ \indiq_{\|X^{1,N}_s-X^{2,N}_s\|\le 1} I^2_s]$ equals 
\begin{align*}
 \gamma (N-2) \E [  \indiq_{\|X^{1,N}_s-X^{2,N}_s\|\le 1}(K(X^{1,N}_s-X^{3,N}_s) +& K(X^{3,N}_s-X^{2,N}_s))\\
& \cdot (X^{1,N}_s-X^{2,N}_s) \|X^{1,N}_s-X^{2,N}_s\|^{\gamma-2}],
\end{align*}
so that using exchangeability again, 
\begin{align}\label{uuu}
 \frac{\theta}{N}\int_0^t\E\Big[ \indiq_{\|X^{1,N}_s-X^{2,N}_s\|\le 1}I^2_s \Big]\dd s = \frac{\gamma \theta (N-2)}{3N} \int_0^t \E [ F(X^{1,N}_s, X^{2,N}_s, X^{3,N}_s) ] \dd s, 
\end{align}
where 
\begin{align*}
F(x,y,z)=& \Big(\frac{Y}{\|Y\|^2}+\frac{Z}{\|Z\|^2}\Big) \cdot X \|X\|^{\gamma -2} \indiq_{\|X\|\le 1} \\
&+ \Big(\frac{Z}{\|Z\|^2}+\frac{X}{\|X\|^2}\Big) \cdot Y \|Y\|^{\gamma -2}\indiq_{\|Y\|\le 1}\\
&+  \Big(\frac{X}{\|X\|^2}+\frac{Y}{\|Y\|^2}\Big) \cdot Z \|Z\|^{\gamma -2}\indiq_{\|Z\|\le 1},
\end{align*}
where we have set $X=x-y$, $Y=y-z$ and $Z=z-x$. By furthermore setting $G(x,y,z) = \|X\|^{\gamma-2} \indiq_{\|X\|\le 1} + \|Y\|^{\gamma-2} \indiq_{\|Y\|\le 1}+\|Z\|^{\gamma-2} \indiq_{\|Z\|\le 1}$, we get that $G(x,y,z)+F(x,y,z)$ equals
\begin{align*}
 \Big( \frac{X}{\|X\|^2}+\frac{Y}{\|Y\|^2}+\frac{Z}{\|Z\|^2}\Big)\cdot \Big(  X\|X\|^{\gamma-2} \indiq_{\|X\|\le 1} + Y\|Y\|^{\gamma-2} \indiq_{\|Y\|\le 1}+Z\|Z\|^{\gamma-2} \indiq_{\|Z\|\le 1} \Big),
\end{align*}
which is positive according to Lemma \ref{inegalitebarycentre}. Injecting in \eqref{uuu}, we get 
\begin{align*}
 \frac{\theta}{N}\int_0^t\E\Big[ \indiq_{\|X^{1,N}_s-X^{2,N}_s\|\le 1}I^2_s& \Big]\dd s \ge -\frac{\gamma \theta (N-2)}{3N} \int_0^t \E [ G(X^{1,N}_s, X^{2,N}_s, X^{3,N}_s) ] \dd s\\
 =& -\frac{\gamma \theta (N-2)}{N} \int_0^t \E [ \indiq_{\|X^{1,N}_s-X^{2,N}_s\|\le 1} \|X^{1,N}_s-X^{2,N}_s\|^{\gamma-2} ] \dd s.
\end{align*}
Gathering this and \eqref{rrr} we get 
\begin{align*}
S_t \ge S_0 + \gamma(\gamma-\theta)\int_0^t \E [ \indiq_{\|X^{1,N}_s - X^{2,N}_s \|\le 1 } \|X^{1,N}_s - X^{2,N}_s \|^{\gamma-2} ]\dd s.
\end{align*}
Since $S_t$ is bounded by $1$ (because $\|X^{1,N}_t-X^{2,N}_t\|^\gamma \wedge 1 \le 1$ a.s), $\gamma \in (\theta,2)$, and $S_0$ is nonnegative, we get that 
$$
\E\Big[\int_0^t  \indiq_{\|X^{1,N}_s - X^{2,N}_s \|\le 1 } \|X^{1,N}_s - X^{2,N}_s \|^{\gamma-2} \dd s \Big]\leq \frac{1}{\gamma(\gamma-\theta)}.
$$
This provides some bound saying that uniformly in $N$, the particles do not spend too much time
close to each other. As we will see, this is sufficient to conclude in the subcritical case.
The critical case is of course more complicated.

\subsection{Plan of the paper}
In Section \ref{basicproperties} we show some basic formulas verified by a $KS$-process. In Section \ref{sectcompact} we will prove Theorem \ref{thm1}-(i) and Theorem \ref{thm2}-(i) and we give the complete proof of Theorem \ref{thm1}-(ii) in Section \ref{subcritical}. Finally, in Section \ref{building} we show some probabilistic results on the behavior of collisions in the critical case which is crucial for the proof of Theorem \ref{thm2}-(ii) in the critical case that is presented in Section \ref{critical}.

\section{The particle system and some basic properties}\label{basicproperties}

This section is devoted to proving the following preliminary results about $KS$-processes.
\begin{lemma}\label{ff}
Let $N\geq 2$, $\theta>0$ and $(\cF_t)_{t\ge 0} $ be a filtration. If $\tau$ is a stopping time for the filtration $(\cF_t)_{t\ge 0} $, then for all $KS(\theta,N)$-process $(X^{i,N}_t)_{t\in [0,\tau), i\in \ig 1,N \id}$ on $[0,\tau)$, for all sequence of increasing $(\cF_t)_{t\ge 0}$-stopping times $(\tau_n)_{n\in \nn }$ such that $\lim_{n\to \infty}\tau_n = \tau$ a.s., all $t\ge 0$, all $n\in \nn$, we have the following identities.

\vip

(i) For all $i \in \ig 1,N\id$, all $\varphi \in C^2(\rr_+)$,
\begin{align*}
\varphi (\| X^{i,N}_{t\wedge \tau_n} \|^2) =& \varphi ( \|X^{i,N}_0 \|^2 ) + 2\int_0^{t\wedge \tau_n} \varphi ' (\|X^{i,N}_s\|^2) X^{i,N}_s \cdot \dd B^i_s \\
&+ 2 \int_0^{t\wedge \tau_n} [\varphi ' (\|X^{i,N}_s\|^2) + \|X^{i,N}_s\|^2 \varphi ''( \|X^{i,N}_s\|^2)] \dd s \nonumber \\
&+\!\! \frac{2\theta}{N}\!\! \int_0^{t\wedge \tau_n}\!\! \sum_{j=1 }^N K(X^{i,N}_s-X^{j,N}_s)\! \cdot\! X^{j,N}_s\! \varphi '(\| X^{j,N}_s\|^2)\dd s.
\end{align*}
(ii) For all $i,j\in \ig 1,N \id $ such that $i\ne j$, all $\varphi \in C^2(\rr_+)$,
\begin{align*}
\varphi ( \|X^{i,N}_{t\wedge \tau_n}-X^{j,N}_{t\wedge \tau_n}\|^2 ) =& \varphi ( \|X^{i,N}_0-X^{j,N}_0 \|^2) + M_t^{N,n} + I^{N,n}_t + \frac{2\theta}N J^{N,n}_t,
\end{align*}
where 
\begin{gather*}
 M^{N,n}_t =  2\int_0^{t\wedge \tau_n} \varphi '(\|X^{i,N}_s-X^{j,N}_s\|^2)(X^{i,N}_s-X^{j,N}_s) \cdot \dd ( B^i_s-B^j_s), \\
I^{N,n}_t = 4\int_0^{t\wedge \tau_n} \Big[\varphi '(\| X^{i,N}_s-X^{j,N}_s\|^2) +\| X^{i,N}_s-X^{j,N}_s\|^2 \varphi ''(\| X^{i,N}_s-X^{j,N}_s\|^2)\Big]   \dd s, \nonumber 
\end{gather*}
and
\begin{align*}
J^{N,n}_t \!\!= \!\!\int_0^{t\wedge \tau_n} \!\!\!\sum_{k=1}^N (K(X^{i,N}_s-X^{k,N}_s)\! &+\! K(X^{k,N}_s-X^{j,N}_s))\\
&\cdot (X^{i,N}_s-X^{j,N}_s)\varphi ' ( \|X^{i,N}_s-X^{j,N}_s \|^2) \dd s.
\end{align*}
(iii) Setting $\mu^N_t = N^{-1} \sum_{i=1}^N \delta_{X^{i,N}_t}$ for all $t\in [0,T)$, we have all $\varphi \in C^2(\rr^2)$,
\begin{align*}
\int_{\rr^2} \varphi (x) \mu^{N}_{t\wedge \tau_n}(\dd x)=& \int_{\rr^2} \varphi (x) \mu^N_0(\dd x) + \frac{1}{N} \int_0^{t\wedge \tau_n} \sum_{i=1}^N\nabla \varphi (X^{i,N}_s)\cdot\dd B^i_s\\
& + \frac12 \int_0^{t\wedge \tau_n} \int_{\rr^2} \Delta \varphi (x) \mu^{N}_s(\dd x) \dd s \nonumber \\
&+ \frac\theta 2 \int_0^{t\wedge \tau_n} \int_{\rr^2} \int_{\rr^2}  K(x-y)\cdot [\nabla \varphi (x) - \nabla \varphi (y)] \mu ^{N}_s(\dd x) \mu^{N}_s(\dd y) \dd s  .
\end{align*}
\end{lemma}
\begin{proof}
First, (i) follows directly from the It\^ o formula, knowing that setting $\psi : x \mapsto \varphi (\|x\|^2)$, we have $\nabla \psi (x) = 2x \varphi '(\|x\|^2)$ and $\Delta \psi (x) = 4 \Big( \varphi '(\|x\|^2) +  \|x\|^2 \varphi ''(\|x\|^2)\Big)$. Moreover (ii) is true for the same reason so it only remains to show (iii).

\vip

 We fix $\varphi \in C^2(\rr^2)$. Applying the It\^ o formula to $\varphi (X^{i,N}_{t\wedge \tau_n})$ and summing over $i \in \ig 1,N \id$, we get
\begin{align}\label{aaa}
\frac1N \sum_{i=1}^N\varphi ( X^{i,N}_{t\wedge \tau_n}) =& \frac{1}{N} \sum_{i=1}^N \varphi ( X^{i,N}_0) + \frac{1}N \sum_{i=1}^N \int_0^{t\wedge \tau_n} \nabla \varphi (X^{i,N}_s)\cdot \dd B^i_s\\
& + \frac1{2N} \sum_{i=1}^N \int_0^{t\wedge \tau_n}  \Delta \varphi (X^{i,N}_{s})  \dd s \nonumber \\
&+ \frac\theta {N^2} \int_0^{t\wedge \tau_n} \sum_{i,j =1}^N K(X^{i,N}_{s}-X^{j,N}_{s})\cdot \nabla \varphi (X^{i,N}_{s})  \dd s. \nonumber
\end{align}
By symmetrizing, we get that for all $s\in [0,t\wedge \tau_n]$,
$$
 \sum_{i,j =1}^N K(X^{i,N}_{s}-X^{j,N}_{s})\cdot \nabla \varphi (X^{i,N}_{s}) = \frac12  \sum_{i,j =1}^N K(X^{i,N}_{s}-X^{j,N}_{s})\cdot [\nabla \varphi (X^{i,N}_{s})-\nabla \varphi (X^{i,N}_{s})],
$$
which together with \eqref{aaa} imply the result.
\end{proof}

\section{Compactness}\label{sectcompact}

Only in this section, $\theta$ will take any positive values, without any other restriction.
We consider $f_0\in \cP (\rr^2)$. We consider also for all $N\ge 5$, $F^N_0 \in \cP_{sym,1}^* ((\rr^2)^N)$ with associated random variable $(X^{1,N}_0,\dots, X^{N,N}_0)$ such that $ \mu_0^N $ goes weakly to $ f_0$ in probability as $N\to \infty$ where $\mu_0^N = N^{-1}\sum_{i=1}^N \delta_{X^{i,N}_0}$. Applying Theorem \ref{fj1}, there exists a $KS(\theta,N)$-process on $[0,\tau_3^N)$ with initial law $F^N_0$ which is denoted $(X^{i,N}_t)_{t\in [0,\tau_3^N), i\in \ig 1,N\id}$. We recall that $N_0= (1+\lceil 2/(2-\theta)\rceil)\lor 5$ 
if $\theta \in (0,2)$ and $N_0=5$ if $\theta\geq 2$.
We define the empirical measure for all $N\ge N_0$ and all $t\geq 0$
$$
\mu^{N,\beta_N}_t = \mu^N_{t\land \beta_N}=\frac{1}{N}\sum_{i=1}^N \delta_{X^{i,N}_{t\land \beta_N}} ,
$$
where we set 
$$
\beta_N = \infty \quad \hbox{if}\quad \theta\in (0,2) \qquad \hbox{and}\qquad \beta_N = \tau_3^{N,\ell_N}
\quad \hbox{if}\quad \theta\geq 2,
$$
with $(\ell_N)_{N\ge 5}$ some increasing sequence of integers. 
It is clear that $(\mu^{N,\beta_N}_t)_{t\ge 0} \in C([0,\infty),\cP(\rr^2))$ a.s. 
because $(X^{i,N}_t)_{t\in [0,\beta_N), i\in \ig 1,N \id}$ is a.s. continuous.   
This section is devoted to showing the following Theorem \ref{thm1}-(i) and Theorem \ref{thm2}-(i). To this end, we need to prove some preliminary estimates.
\begin{prop}\label{ascolici}
(i) There exists some nondecreasing $\psi \in C(\rr_+)$ such that $\psi (0) = 1$, 
$\psi (2r) \le C\psi (r)$ for some constant $C>0$, $\lim_{r\to \infty } \psi (r) = \infty$ and 
\begin{align*}
\mbox{ for all } T\ge 0, \qquad 
M_T:=\sup_{N\ge N_0} \E \Big[ \sup_{t\in [0,T]}\frac{1}{N}\sum_{i=1}^N\psi(\|X_{t\wedge\beta_N}^{i,N}\|^2) 
\Big] <\infty .
\end{align*}
\vip
(ii) For all $T>0$, $\e>0$, there exists $K_{\e,T}$ a compact set of $\cP (\rr^2)$ such that, 
$$
\mbox{ for all } N\ge N_0, \quad \PP( \forall t\in [0,T], \mu^{N,\beta_N}_t \in K_{\e,T} ) \ge 1-\e.
$$
\end{prop}
\begin{proof}
We divide the proof in three steps.
\vip
{\it Step 1.} We first show that $(\|X^{1,N}_0\|)_{N\ge N_0}$ is tight. Since 
$\mu^N_0$ goes weakly to $f_0$ in probability, we conclude that 
$\limsup_{N\to \infty} \E[\mu^N_0(\{ x \in \rr^2 : ||x||\geq K\}] \leq f_0(\{x \in \rr^2 : ||x||\geq K\})$
for all $K\geq 1$.
Since it holds that $\E[\mu^N_0(\{ x \in \rr^2 : ||x||\geq K\}]=\PP(||X^{1,N}_0||\geq K)$, we conclude that
$\limsup_{K\to \infty} \limsup_{N\to \infty}\PP(||X^{1,N}_0||\geq K)=0$.

\vip

{\it Step 2.}
We conclude the proof of (i). Since $(X^{1,N}_0)_{N\ge N_0}$ is tight, by Lemma \ref{vallepoussin} there exists some nondecreasing $\psi \in C^2 (\rr_+)$ such that $\lim_{r\to \infty} \psi (r) = \infty $, $\psi (0) =1$, 
$\psi (2r)\le C \psi (r)$ and $(1+r)|\psi'(r)| + r|\psi''(r)|\le C$ for some constant $C>0$, and such that
\begin{align}\label{cifini}
\sup_{N\ge N_0} \E [\psi (\|X^{1,N}_0\|^2)] < \infty.
\end{align}
 We fix $T>0$. Applying Lemma \ref{ff}-(iii) and taking the supremum we get,
\begin{align}\label{tttt}
\sup_{t\in [0,T]}\frac1N \sum_{i=1}^N\psi ( \|X^{i,N}_{t\wedge \beta_N}\|^2 ) \le & \frac1N \sum_{i=1}^N\psi ( \|X^{i,N}_0\|^2 )+ I^1_T + I^2_T + I^3_T,
\end{align}
where
\begin{gather*}
I^1_T = \frac2N \sum_{i=1}^N \int_0^{T\wedge \beta_N} 
(|\psi ' (\|X^{i,N}_s\|^2)| + \|X^{i,N}_s\|^2 |\psi ''( \|X^{i,N}_s\|^2) )|\dd s, \\
I^2_T = 2\sup_{t\in [0,T]} \Big(\frac1N \sum_{i=1}^N \int_0^{t\wedge \beta_N} \psi'(\|X^{i,N}_s\|^2) X^{i,N}_s \cdot \dd B^{i}_s \Big),\\
I^3_T = \frac{\theta}{N^2} \sup_{t\in [0,T]}\int_0^{t\wedge \beta_N} \!\sum_{i,j=1}^N K(X^{i,N}_s-X^{j,N}_s) 
\cdot [X^{i,N}_s \psi '(\| X^{i,N}_s\|^2)-X^{j,N}_s \psi '(\| X^{j,N}_s\|^2)] \dd s.
\end{gather*}
where $I^3_T$ was obtained thanks to a symmetrization argument.
Thanks to the hypothesis we made on $\psi$, we have $I^1_T \le 2 C (T\wedge \beta_N)$ a.s. Moreover, since $||K(x)||\leq ||x||^{-1}$, and since $g : x \in \rr^2 \mapsto x \psi' (\|x\|^2)$ is differentiable with a bounded differential according to the properties of $\psi$, we get that $I^3_T \le \theta \|\nabla g\|_{\infty} (T\wedge \beta_N)$ a.s. Finally, using the Cauchy-Schwarz and the Doob inequalities, we find that there exists a universal constant $ C'>0$, of which the value is allowed to vary, such that 
$$
\E[I^2_T] \le C' \Big(\frac1{N^2} \sum_{i=1}^N\E \Big[ \int_0^{T\wedge \beta_N} (\psi'(\|X^{i,N}_s\|^2))^2 \|X^{i,N}_s\|^2 \dd s \Big]\Big)^{1/2} \le C' T^{1/2},
$$ where we used the properties of $\psi$.
This implies, by taking the expectation in \eqref{tttt} that 
\begin{align*}
\E\Big[\sup_{t\in [0,T]}\frac1N \sum_{i=1}^N\psi ( \|X^{i,N}_{t\wedge \beta_N}\|^2 )\Big] \le & \E\Big[\sup_{N\ge N_0} \psi ( \|X^{1,N}_0\|^2 )\Big] + (2 C +\theta \|\nabla g\|_{\infty})T + C' T^{1/2}.
\end{align*}
Combining with \eqref{cifini}, we get the result.

\vip

{\it Step 3.} We fix $T>0$ and $\e>0$ and, for some positive sequences 
$(A^\e_k)_{k\ge 0}$ and $(\eta^\e_k)_{k\ge 0}$ such that $\lim_{k\to \infty} A^\e_k=\infty$ and 
$\lim_{k\to \infty} \eta^\e_k=0$, to be chosen later, we introduce the compact
$$
K_{\e,T}=\bigcap_{k\geq 0} \Big\{\mu \in \cP(\rr^2) : \mu(B(0,A_k^\e)^c)\leq\eta^\e_k\Big\}
$$
of $\cP(\rr^2)$. Now for $N\geq N_0$,
\begin{align*}
\PP ( \exists t \in [0,T] : \mu^{N,\beta_N}_t \notin K_{\e,T}) 
&\le \sum_{k\ge 0} \PP \Big( \exists t\le T, \mu_t^{N,\beta_N}(B(0,A^\e_k)^c)>\eta^\e_k \Big)\\
&= \sum_{k\ge 0} \PP \Big( \exists t\le T,\frac1N \sum_{i=1}^N \indiq_{\|X^{i,N}_{t\wedge \beta_N}\| > A^\e_k} > \eta^\e_k \Big)\\
&\le \sum_{k\ge 0} \PP \Big( \exists t\le T,\frac1N \sum_{i=1}^N \indiq_{\psi (\|X^{i,N}_{t\wedge \beta_N}\|^2) > \psi ((A^\e_k)^2)} > \eta^\e_k \Big)\\
&\le \sum_{k\ge 0} \PP \Big( \exists t\le T,\frac1N \sum_{i=1}^N \psi (\|X^{i,N}_{t\wedge \beta_N}\|^2)  > \eta^\e_k \psi ( (A^\e_k)^2) \Big)\\
&\le \sum_{k\ge 0}\frac{\E \Big(\sup_{t\in [0,T]} \frac1N \sum_{i=1}^N \psi (\|X^{i,N}_{t\wedge \beta_N}\|^2) \Big) }{\eta^\e_k \psi ( (A^\e_k)^2)}\\
&\leq  M_T \sum_{k\ge 0}\frac1{\eta^\e_k \psi ( (A^\e_k)^2)}
\end{align*} 
by (i). We now choose $A^\e_k$ such that $\sum_{k\ge 0} 1/\sqrt{ \psi ( (A^\e_k)^2)} \le \e/M_T$ and 
$\eta^\e_k = 1/ \sqrt{ \psi ( (A^\e_k)^2)}$. We find that 
$\PP ( \exists t \in [0,T] : \mu^{N,\beta_N}_t \notin K_{\e,T}) \leq \e$ as desired.

\end{proof}
\begin{proof}[Proof of Theorem \ref{thm1}-(i) and Theorem \ref{thm2}-(i)]
We recall that one can find a sequence $(\varphi_n)_{n\in \nn}$ of elements of $C^2(\rr^2)$ such that 
$$
\|\varphi_n\|_\infty + \|\nabla \varphi_n\|_\infty + \|\nabla^2\varphi_n\|_\infty \le 1 \quad \mbox{ for all } n\ge 0,
$$
and such that the distance $\delta$ defined for all $f$, $g\in \cP (\rr^2)$ by
$$
\delta (f,g) = \sum_{n\ge 0} 2^{-n} \Big|\int \varphi_n (x) f (\dd x) - \int \varphi_n (x) g (\dd x) \Big|,
$$
metrizes the weak convergence topology in $\cP (\rr^2)$. Since $C([0,\infty),\cP (\rr^2))$ is endowed with the topology of the uniform convergence on compact sets, it is sufficient to prove the tightness of $((\mu_t^N)_{t\ge 0})_{N\ge N_0}$ on $C([0,T],\cP (\rr^2))$ for every $T>0$. We fix $T>0$ and set the following compact set of $C([0,T], \cP (\rr^2))$,
 $$
\cK _{A,K} := \Big\{ (f_t)_{t\ge 0} \in C([0,T],\cP(\rr^2)) :\forall t \in [0,T], f_t\in K \mbox{ and }  \sup_{\substack{s,t \in [0,T]\\ s\ne t}}\frac{\delta (f(t),f(s))}{|t-s|^{1/4}} \le A\Big\},
$$
for all $A>0$ and for all $K$ compact set of $\cP (\rr^2)$. It suffices to show
\begin{align}\label{cond}
\sup_{N\ge N_0}\E \Big[\sup_{\substack{s, t \in [0,T]\\ s\ne t}} \frac{\delta (\mu^{N,\beta_N}_t,\mu^{N,\beta_N}_s)}{|t-s|^{1/4}}\Big]< \infty.
\end{align}
Indeed, if it is the case, thanks to the Markov inequality, for all $\e >0$, there exists $A_\e >0$ such that 
$$
\mbox{ for all } N\ge N_0, \quad \PP \Big( \sup_{\substack{s,t \in [0,T]\\ s\ne t}}\frac{\delta (\mu^{N,\beta_N}_t,\mu^{N,\beta_N}_s)}{|t-s|^{1/4}} \le A_\e \Big)\ge 1-\e,
$$ 
and using $(K_{\e,T})_{\e >0}$ defined in Proposition \ref{ascolici}-(ii), we have for all $N\ge N_0$,
$$
\PP ( (\mu^{N,\beta_N}_t)_{t\ge 0} \in \cK_{A_\e,K_{\e,T}}) \ge 2(1-\e) - 1 =1-2\e,
$$
which ends the proof. Let us show \eqref{cond}. Using Lemma \ref{ff}-(iii) and the fact that $\|K(x)\| \le \|x\|^{-1}$ for $x\ne 0$, we have for all $n\ge 0$, all $T>0$, all $s,t\in [0,T]$ such that $s\ne t$,

\begin{align*}
\frac1N\Big|\sum_{i=1}^N (\varphi_n (X^{i,N}_{t\wedge \beta_N})- \varphi_n (X^{i,N}_{s\wedge \beta_N}))\Big|&\le S^1 + S^2,
\end{align*}
where
\begin{gather*}
S^1 = \frac1N \Big|\sum_{i=1}^N\int_{s\wedge \beta_N}^{t\wedge \beta_N} \nabla \varphi_n (X^{i,N}_u)\dd B^{i}_u\Big|,\\
S^2 = \! \frac1N \sum_{i=1}^N\int_{s\wedge \beta_N}^{t\wedge \beta_N} \frac12 |\Delta \varphi_n (X^{i,N}_u)|\dd u +\!\!\frac\theta {2N} \sum_{\substack{i,j=1\\j\ne i}}^N \int_{s\wedge \beta_N}^{t\wedge \beta_N} \frac{\|\nabla \varphi_n(X^{i,N}_u) -\nabla \varphi_n(X^{j,N}_u)\|}{\|X^{i,N}_u-X^{j,N}_u\|} \dd u .
\end{gather*}
Since $\nabla^2 \varphi_n$ is bounded, $S^2$ is bounded by $C(t\wedge \beta_N-s\wedge \beta_N) \le C(t-s)$ with 
$C>0$ a constant. Moreover, by the Dubins-Schwarz theorem, even if it means to enlarge the probability space, there exists a Brownian motion $(\tilde{B}_t)_{t\ge 0}$ such that 
$$
\sum_{i=1}^N\int_0^{t\wedge \beta_N} \nabla \varphi_n(X^{i,N}_u) \dd B^i_u = \tilde B _{\int_0^{t\wedge \beta_N}\sum_{i=1}^N\|\nabla \varphi_n(X^{i,N}_u)\|^2\dd u }.
$$
Recalling that $||\nabla \varphi_n||_\infty \leq 1$ and setting $C_{\tilde B} : =\sup_{u,v\in [0,T], u\ne v} (\tilde B_u-\tilde B_v)/|u-v|^{1/4}$, this implies that 
\begin{align*}
S^1 &\le N^{-1}C_{\tilde B} \Big(\sum_{i=1}^N \int_{s\wedge \beta_N}^{t\wedge \beta_N} \| \nabla \varphi_n( X^{i,N}_u)\|^2 \dd u\Big)^{1/4}\\
&\le C_{\tilde B}N^{-3/4}(t\wedge\beta_N-s\wedge \beta_N)^{1/4}\\
&\le C_{\tilde B} N^{-3/4}(t-s)^{1/4}.
\end{align*}

But $\E[C_{\tilde B}]<\infty$ thanks to the Kolmogorov criterion, 
so that recalling the definition of $\delta$,
\begin{align*}
\E\Big[\sup_{\substack{s,t\in [0,T]\\s\ne t}}\frac{\delta ( \mu^{N,\beta_N}_t, \mu^{N,\beta_N}_s )}{|t-s|^{1/4}} \Big]
&\le \sum_{n\geq 0} 2^{-n} \E\Big[\frac1N\Big|\sum_{i=1}^N(\varphi_n (X^{i,N}_{t\wedge \beta_N})- \varphi_n (X^{i,N}_{s\wedge \beta_N}))\Big|\Big]\\
&\le \E[2(CT^{3/4}+N^{-3/4}C_{\tilde B})].
\end{align*}
This proves \eqref{cond}. 
\end{proof}

\section{The subcritical case}\label{subcritical}

In this section, we prove Theorem \ref{thm1}-(ii). We recall that in this case $\theta \in (0,2)$. 
We begin with a simple geometrical result which is crucial for our purpose, that we proved in \cite{ft}
but that we recall for completeness.
\begin{lemma}\label{inegalitebarycentre}
For all pair of nonincreasing functions $\varphi, \psi : (0,\infty) \to \rr_+$, for all $X,Y,Z \in \rr^2\setminus \{0\} $ 
such that $X+Y+Z=0$, we have 
$$
\Delta = [ \varphi (\|X\|)X + \varphi (\|Y\|)Y + \varphi(\|Z\|)Z ] \cdot [\psi (\|X\|)X 
+ \psi (\|Y\|)Y + \psi(\|Z\|)Z ] \ge 0.
$$
More precisely, if e.g. $\|X\|\leq\|Y\|\leq\|Z\|$, we have
$$
\Delta \geq [\varphi (\|X\|)-\varphi (\|Y\|)][\psi(\|X\|)-\psi(\|Y\|)]\|X\|^2.
$$
\end{lemma}
\begin{proof}
We may study only the case where $\|X\| \le \|Y\| \le \|Z\|$. Since $Y=-X-Z$,
\begin{gather*}
\varphi (\|X\|)X + \varphi (\|Y\|)Y + \varphi(\|Z\|)Z = \lambda X - \mu Z, \\
\psi (\|X\|)X + \psi (\|Y\|)Y + \psi(\|Z\|)Z = \lambda' X - \mu' Z,
\end{gather*}
where $\lambda = \varphi (\|X\|)-\varphi (\|Y\|) \ge 0$, 
$\mu = \varphi (\|Y\|)-\varphi (\|Z\|) \ge 0$, 
$\lambda' = \psi (\|X\|)-\psi (\|Y\|) \ge 0$ and $\mu' = \psi (\|Y\|)-\psi (\|Z\|) \ge 0$. Therefore,
$$
\Delta = \lambda \lambda' \|X\|^2 + \mu \mu' \|Z\|^2 
- (\lambda \mu' + \lambda' \mu)X\cdot Z\geq \lambda \lambda' \|X\|^2
$$
as desired, because $X\cdot Z\leq 0$. Indeed,
if $X\cdot Z >0$, then
$\|Y\|^2=\|Z+X\|^2 = \|Z\|^2 + \|X\|^2 + 2X\cdot Z > \|Z\|^2\geq \|Y\|^2$, which is absurd. 
\end{proof}

 The next result follows the proof of \cite[Proposition 5]{ft}.
\begin{prop}\label{estimgamma}
For all $\gamma \in (\theta,2)$, there exists a constant $C_{\theta,\gamma} >0$ such that for all $t\ge 0$,
\begin{align*}
\sup_{N\ge N_0}\E \Big[ \int_0^{t}  \|X^{1,N}_s-X^{2,N}_s\|^{\gamma-2} \dd s\Big]\le 
C_{\theta,\gamma} (1+t).
\end{align*}
\end{prop}
\begin{proof}
We fix $a>0$ and set $\varphi_a (r) =  (r+a)^{\gamma/2}/(1+(r+a)^{\gamma/2})$. We first show that
for some constant $C_1>0$, we have
\begin{align}\label{gggg}
\varphi_a'(r) + r \varphi_a''(r)\ge c\frac{\gamma}{2}\varphi_a'(r) - C_1, \qquad \hbox{where}\quad  
c= \frac 12\Big(1 + \frac\theta  \gamma\Big) \in \Big(\frac\theta\gamma,1\Big).
\end{align}
We start with
$$
\varphi_a'(r) = \frac\gamma 2 \frac{(r+a)^{\gamma/2-1}}{(1+(r+a)^{\gamma/2})^2}\,\, \mbox{ and }\,\, \varphi_a''(r)\! = \frac{\gamma}{2}\frac{(r+a)^{\gamma/2-2}}{(1+(r+a)^{\gamma/2})^2} \Big(\frac\gamma 2 -1 -\gamma \frac{(r+a)^{\gamma/2}}{1+(r+a)^{\gamma/2}}\Big).
$$
This implies since $a\ge (\gamma/2) a$,
\begin{align*}
\varphi_a'(r) + r \varphi_a''(r) &= \varphi_a'(r) \Big(   \frac{\frac{\gamma}{2}r+a}{r+a} - \gamma \frac{r(r+a)^{\gamma/2-1}}{1+(r+a)^{\gamma/2}}\Big)\\
&\ge \frac{\gamma} 2 \varphi_a'(r) \Big( 1 - 2\frac{r(r+a)^{\gamma/2-1}}{1+(r+a)^{\gamma/2}}\Big)\\
&\ge \frac{\gamma} 2 \varphi_a'(r) \Big( 1 - 2\frac{(r+a)^{\gamma/2}}{1+(r+a)^{\gamma/2}}\Big).
\end{align*}
We now set $u = (r+a)^{\gamma/2}$. When $0\le u \le (1-c)/(3+c)$, it holds that $(1-c) - 2u/(1+u) \ge (1-c)/2$,
whence
$$
\varphi_a'(r) + r\varphi_a''(r) \ge \frac\gamma 2 \varphi_a'(r) \Big( c+\frac{1-c}{2}\Big) \ge \frac{c\gamma} 2 \varphi_a'(r).
$$
Otherwise if $u> (1-c)/(3+c)$ i.e. $(r+a) > v = ((1-c)/(3+c))^{2/\gamma}$, we use the crude estimate
$$
1-c - 2\frac{u}{1+u} \ge -c-1,
$$
so that
$$
\varphi_a'(r) + r\varphi_a''(r) \ge - (1+c)\frac{\gamma}{2}\varphi_a'(r).
$$
Since $\psi : x\mapsto x^{\gamma/2-1}/(1+x^{\gamma/2})^2$ is bounded on every compact set excluding $0$, there exists a constant $C_1>0$ only depending on $\theta$ and $\gamma$ such that, still when $u> (1-c)/(3+c)$,
$$
\varphi_a'(r) + r\varphi_a''(r) \ge c\frac\gamma 2\varphi'_a(r)- C_1.
$$
We have proved \eqref{gggg}.
\vip
Applying Lemma \ref{ff}-(ii) to $\varphi_a(\|X^{1,N}_t-X^{2,N}_t\|^2)$, using \eqref{gggg} and taking the expectation, we find
\begin{align}\label{s123}
\E ( \varphi_a( \| X^{1,N}_{t} -X^{2,N}_{t}\|^2)) &\ge \E ( \varphi_a(\| X^{1,N}_0-X^{2,N}_0\|^2)) + \Big(2c\gamma -\frac{4\theta}{N} \Big) S^1_t  -\frac{2\theta}{N} S^2_t -4C_1t,
\end{align}
where
\begin{align}\label{aabb}
S^1_t = \E\Big[ \int_0^{t}   \varphi_a'(\|X^{1,N}_s-X^{2,N}_s\|^2) \dd s\Big]
\end{align}
and
\begin{align*}
S^2_t = \E \Big[\!\int_0^{t} \!\!\!\varphi_a'(\|X^{1,N}_s-X^{2,N}_s\|^2)& (X^{1,N}_s-X^{2,N}_s)\\
&\cdot \!\!\sum_{\substack{i=1\\i\ne 1,2}}^N\!\!\Big( \frac{X^{1,N}_s-X^{i,N}_s}{\|X^{1,N}_s-X^{i,N}_s\|^2} \!+\!\frac{X^{i,N}_s-X^{2,N}_s}{\|X^{i,N}_s-X^{2,N}_s\|^2}\Big)\dd s \Big].
\end{align*}
By exchangeability, we get that $S^2_t/(N-2)$ equals
\begin{align*}
 \E \Big[\int_0^{t} \!\!\!\varphi_a'(\|X^{1,N}_s-X^{2,N}_s\|^2 )(X^{1,N}_s-X^{2,N}_s)\!\cdot\! \Big( \frac{X^{1,N}_s-X^{3,N}_s}{\|X^{1,N}_s-X^{3,N}_s\|^2}\!+\!\!\frac{X^{3,N}_s-X^{2,N}_s}{\|X^{3,N}_s-X^{2,N}_s\|^2}\Big) \dd s \Big]\nonumber\\
&
\end{align*}
so that by exchangeability again,
\begin{align}
\frac{S^2_t}{N-2} =-\frac{1}{3} \E \Big[\int_0^{t} F_a(X^{1,N}_s,X^{2,N}_s,X^{3,N}_s)\dd s \Big], \label{www}
\end{align}

where 
\begin{align*}
F_a (x,y,z) =& \varphi_a'(\|X\|^2)X\cdot \Big(\frac{Z}{\|Z\|^2} + \frac{Y}{\|Y\|^2}\Big)\\
& + \varphi_a'(\|Y\|^2)Y\cdot \Big(\frac{X}{\|X\|^2} + \frac{Z}{\|Z\|^2}\Big) \\
&+ \varphi_a'(\|Z\|^2)Z\cdot \Big(\frac{Y}{\|Y\|^2} + \frac{X}{\|X\|^2}\Big),
\end{align*}
with $X=x-y$, $Y=y-z$ and $Z=z-x$. We now introduce $G_a(x,y,z)=\varphi_a'(\|X\|^2)+ \varphi_a'(\|Y\|^2) + \varphi_a'(\|Z\|^2)$ and observe that for all $X$, $Y$, $Z \in \rr^2\setminus \{ 0 \}$,
$$
G_a(x,y,z) =\varphi_a'(\|X\|^2)X\cdot\frac{X}{\|X\|^2}
+\varphi_a'(\|Y\|^2)Y\cdot\frac{Y}{\|Y\|^2}
+\varphi_a'(\|Z\|^2)Z\cdot\frac{Z}{\|Z\|^2}.
$$
Hence for all $X$, $Y$, $Z\in \rr^2\setminus \{0\}$, $G_a(x,y,z)+F_a(x,y,z)$ equals
\begin{align*}
\Big(\varphi_a'(\|X\|^2)X
+\varphi_a'(\|Y\|^2)Y+\varphi_a'(\|Z\|^2)Z\Big)\cdot \Big(\frac{X}{\|X\|^2}+\frac{Y}{\|Y\|^2}+\frac{Z}{\|Z\|^2}&\Big),
\end{align*}
which is nonnegative according to Lemma \ref{inegalitebarycentre}, since $r\to r^{-2}$
and $r\to \varphi_a'(r)$ are both nonnegative and nonincreasing on $(0,\infty)$
and $X+Y+Z=0$. Thus $F_a(x,y,z)\geq -G_a(x,y,z)$, which injected in \eqref{www} gives
\begin{align}\label{xyz}
S^2_t &\le \frac{N-2 }{3} \E \Big[\int_0^{t} G_a(X^{1,N}_s,X^{2,N}_s,X^{3,N}_s)\dd s \Big]\nonumber\\
&=(N-2) \E \Big[\int_0^{t} \varphi_a'(\|X^{1,N}_s-X^{2,N}_s\|^2)\dd s \Big],
\end{align}
by exchangeability and since $\int_0^{t} \indiq_{\{X^{1,N}_s=X^{2,N}_s\}} \dd s = 0$ a.s. according to Theorem \ref{fj1}. Gathering \eqref{s123}, \eqref{aabb} and \eqref{xyz}, the fact that $\varphi_a$ is nonnegative and bounded by $1$ and recalling the expression  of $\varphi_a'$, we get
\begin{align*}
1 + 4C_1t \ge& 2 (c\gamma - \theta)\E \Big[\int_0^{t} \frac{(\|X^{1,N}_s-X^{2,N}_s\|^2+a)^{\gamma/2-1}}{(1+ \|X^{1,N}_s-X^{2,N}_s\|^2+a)^{\gamma/2})^2} \dd s \Big].
\end{align*}
Recalling that $c=(1+\theta/\gamma)/2$ we get that $c\gamma - \theta =(\gamma -\theta)/2>0,$ so that
\begin{align*}
\E \Big[\int_0^{t} \frac{(\|X^{1,N}_s-X^{2,N}_s\|^2+a)^{\gamma/2-1}}{(1+ \|X^{1,N}_s-X^{2,N}_s\|^2+a)^{\gamma/2})^2} \dd s \Big] \le \frac{1+4C_1t}{\gamma - \theta}. 
\end{align*}
Recalling that $\gamma/2-1 <  0$ and  $\gamma >0$, we find  
$x^{\gamma /2 -1} \le 2 x^{\gamma/2-1}/( 1+x^{\gamma/2})^2 + 1$ (distinguish the cases where $x>1$ and $x\le 1$). 
This implies that
\begin{align*}
\E \Big[\int_0^{t} (\|X^{1,N}_s-X^{2,N}_s\|^2+a)^{\gamma/2-1} \dd s \Big] \le \frac{2(1+4C_1t)}{\gamma - \theta} + t,
\end{align*}
so that the monotone convergence theorem  (letting $a\to 0$)  completes the proof.
\end{proof}
\begin{proof}[Proof of Theorem \ref{thm1}-(ii)]
Recall that $\theta \in (0,2)$, $f_0 \in \cP(\rr^2)$ and that $(\mu^{N}_t)_{t\geq 0}$ is the empirical process
associated to a solution $(X^{i,N}_t)_{t\geq 0,i\in \ig 1,N\id}$ to \eqref{EDS}, for each $N\geq N_0$.
By Theorem \ref{thm1}-(i), we know that the family $((\mu^{N}_t)_{t\geq 0},N\geq N_0)$ is tight in
$C([0,\infty),\cP(\rr^2))$. We now consider a sequence $(N_k)_{k\ge 0}$ and a random variable
$(\mu_t)_{t\ge 0}\in C([0,\infty),\cP(\rr^2))$ such that $\lim_k N_k=\infty$ and
$(\mu^{N_k}_t)_{t\ge 0}$ goes to $(\mu_t)_{t\ge 0}$ in law as $k\to \infty$. 
We have $\mu_0=f_0$ since by hypothesis, $\mu^N_0$ goes weakly to $f_0$ in probability as $N\to \infty$.

\vip

{\it Step 1.}  We prove  here  
that $\E[\int_0^t \|x-y\|^{\gamma -2} \mu_s(\dd x) \mu_s(\dd y)\dd s ]<\infty$.
First, by exchangeability, we get that for all $M>0$,
\begin{align*}
 \E \Big(\int_0^{t}\!\int_{\rr^2} \!\int_{\rr^2} \!(\|x-y\|^{\gamma-2}\wedge\! M)\mu^{N_k}_s(\dd x)& \mu^{N_k}_s(\dd y) \dd s \Big)\! \leq \frac{M t}{N_k}\! \\
 &+\!  \frac{N_k-1}{N_k} \E \Big(\int_0^{t} \|X^{1,N_k}_s-X^{2,N_k}_s\|^{\gamma-2}\dd s \Big).
\end{align*}
According to Proposition \ref{estimgamma}, since $(\mu^{N_k}_t)_{t\ge 0}$ goes to $(\mu_t)_{t\ge 0}$ in law as $k\to \infty$, there exists  a constant $C_{\theta,\gamma}>0$ only depending on $\theta$ and $\gamma$  such that 
$$
\E \Big(\int_0^t\int_{\rr^2} \int_{\rr^2} (\|x-y\|^{\gamma-2}\wedge M)\mu_s(\dd x) \mu_s(\dd y)\dd s \Big) \le C_{\theta,\gamma}(1+T).
$$
Since this holds true for every $M>0$, the monotone convergence theorem gives us the result by letting $M\to \infty$.

\vip

{\it Step 2.} It only remains to check that $\mu$ is a.s. a weak solution to \eqref{EDP}. 
We apply Lemma \ref{ff}-(iii) to $\varphi \in C^2_b(\rr^2)$   and get 
$I_1(\mu^{N_k}) -I_2(\mu^{N_k})= M_k(t)$, where 
\begin{align*}
M_k(t) = \frac1{N_k} \int_0^t \sum_{i=1}^{N_k} \nabla \varphi (X^{i,N_k}_s)\cdot\dd B^i_s,
\end{align*}
and for all $\nu \in C([0,\infty), \cP(\rr^2))$, 
\begin{gather*}
I^1_t(\nu) = \int_{\rr^2} \varphi (x) \nu_t(\dd x)-\int_{\rr^2} \varphi (x) \nu_0(\dd x)-\frac12 \int_0^t \int_{\rr^2} \Delta \varphi (x) \nu_s(\dd x)\dd s,\\
I^2_t(\nu)=\frac\theta2\int_0^t \int_{\rr^2}\int_{\rr^2}  K(x-y)\cdot [\nabla\varphi (x)-\nabla \varphi(y)] 
\nu_s(\dd x) \nu_s (\dd y) \dd s.
\end{gather*}
Since $\varphi$ and $\Delta \varphi$ are continuous  and bounded, $I^1_t(\nu)$ is clearly continuous with respect to $\nu$ and bounded. Since $g:(x,y)\in (\rr^2)^2 \mapsto K(x-y) \cdot [\nabla \varphi (x)-\nabla \varphi (y)] \indiq_{x\ne y}$ is continuous and bounded on $(\rr^2)^2\setminus D$ where $D=\{(x,y)\in (\rr^2)^2 : x= y \}$, $I^2_t(\nu)$ is also continuous $\cL ( (\mu_t)_{t\ge 0})$-a.e, where $\cL ( (\mu_t)_{t\ge 0})$ is the law of $(\mu_t)_{t\ge 0}$. Indeed, Step 1 implies that for all $t\ge 0$, $\int_0^t(\mu_s\otimes \mu_s) ( D )\dd s = 0$ a.s. Thus, $I^1_t(\mu^{N_k}) -I^2_t(\mu^{N_k})$ goes in law to $I^1_t(\mu) -I^2_t(\mu)$ as $k\to \infty$. 

\vip

Moreover, $\lim_{k\to \infty}\E [ (I^1_t(\mu^{N_k}) -I^2_t(\mu^{N_k}))^2] = \lim_{k\to \infty}\E [M_k(t)^2] =0$, which implies $I^1_t(\mu^{N_k}) -I^2_t(\mu^{N_k})$ goes to $0$ in probability as $k\to \infty$. We conclude that for each $t\geq 0$,
$$
I^1_t(\mu) = I^2_t(\mu) \quad \mbox{ a.s. }
$$ 
This implies the result, i.e. that a.s., for all $t\geq 0$, $I^1_t(\mu) = I^2_t(\mu)$ 
by continuity,
since $(\mu_t)_{t\ge 0} \in C([0,\infty), \cP (\rr^2))$, since $ \varphi \in C^2_b$ and since
$g$ is bounded and for all $s\ge 0$, $\mu_s \in \cP (\rr^2)$.
\end{proof}

\section{Estimation of the first triple collision time}\label{building}

Since the existence of the particle system is not guaranteed after the time of the first triple collision, meaning the first collision between  at least  three particles (which a.s. occurs in the case where $\theta > 2(N-2)/(N-1)$ according to Theorem \ref{fj1}), we need to show that this time goes to the infinite as the number of particle increases to the infinite. This work has been more or less done  in the paper of Fournier-Tardy \cite{ftc},  but since the interest in \cite{ftc} was to study the $KS$-process near the blow-up instant, and  since  there is no solution to \eqref{EDS} in the classical sense in this region of the time, the Dirichlet forms theory  was  used  to provide a sense to such a process.  
This makes the proofs of the paper \cite{ftc} very complicated. 

\vip

We define for all $\theta >0$, all $N\ge 2$,

$$ 
 d_{\theta,N}(k) = (k-1)\Big( 2 - \frac{k\theta}{N}\Big) \mbox{ for all } k\in \ig 2,N\id \quad \mbox{ and } \quad k_2 = \min \{ k\ge 3 : d_{2,N}(k) <2\}.
$$

This section is devoted to proving the following result. Recall the definitions of $\tau_k^N$ and $\tau_k^{N,\ell}$ for $N\ge 2$, $k\in \ig 2,N \id$ and $\ell \ge 1$, see Subsection \ref{criandsupercri}. 
\begin{prop}\label{fj2}
If $\theta = 2$, we consider for all $N\ge N_0$,  $F^N_0 \in \cP_{sym,1}^*(\rr^2)$ and a $KS(\theta,N)$-process $(X^{i,N}_t)_{t\in [0,\tau_3^N), i\in \ig 1,N \id}$ on $[0,\tau_3^N)$ with initial law $F^N_0$, then 
 
 \vip
 
(i) $ \tau_3^N = \tau_{k_2}^N $ a.s.

\vip

(ii) There exists an increasing sequence of deterministic integers $(\ell_N)_{N\ge 5}$ such that setting $\beta_N = \tau_{3}^{N,\ell_N}$, then for all $t\ge 0$, $\PP (\beta_N \le t ) \to _{N\to \infty} 0$.
\end{prop}
Point (i) explains that the first collision involving strictly more than two particles is a collision between at least $k_2$ particles. Moreover, point (ii) means that the instant of this collision tends to infinity in probability as the number of particles in the system tends to infinity.

\vip

We begin with some preliminary results about the behavior of the empirical variance of subsets of the particles. The following has been essentially treated  in the paper of Fournier-Jourdain \cite{fj}. Recall the definition of $R_{\ig 1,N \id }$, see Subsection \ref{criandsupercri}, the definition of a squared Bessel process with dimension $\delta \in \rr$  and the fact that if $(Z_t)_{t\ge 0}$ is a squared Bessel process with dimension $\delta$,
\begin{itemize}
\item if $\delta \ge 2$, then a.s. for all $t>0$, $Z_t >0$,
\item if $\delta \in (0,2)$, then a.s. $(Z_t)_{t\ge 0}$ hits $0$ in finite time and is instantaneously reflected,
\item if $\delta \le 0$, then a.s. there exists $T>0$ such that for all $t\ge T$, $Z_t = 0$,  
\end{itemize}
see  the book of Revuz-Yor  \cite[Chapter XI]{ry}.

\begin{lemma}\label{bessel}
Let $(\cF_t)_{t\ge 0}$ be a filtration, $\tau$ a $(\cF_t)_{t\ge 0} $-stopping time such that $\tau \le \tau^N_N$ a.s. and $(\tau_n)_{n\ge 0}$ a sequence of increasing $(\cF_t)_{t\ge 0}$-stopping times such that $\lim_{n\to \infty} \tau_n = \tau$. If $\theta >0$ and $N\ge 5$, then for all $KS(\theta,N)$-process $(X^{i,N}_t)_{ t\in [0,\tau), i\in \ig 1,N \id}$ on $[0,\tau)$, $(R_{\ig 1,N \id }(X^N_t))_{t\in [0,\tau)}$ is a $(\cF_t)_{t\ge 0}$-measurable squared Bessel process with dimension $d_{\theta,N}(N)$ restricted to $[0,\tau)$.
\end{lemma}

\begin{proof}
We apply Lemma \ref{ff}-(ii) with $\varphi = I_d$ and we get for all $i\ne j\in \ig 1,N \id $, all $n\ge 0$,
\begin{align}\label{xcv}
 \|X^{i,N}_{t\wedge \tau_n}-X^{j,N}_{t\wedge \tau_n}\|^2  =&\|X^{i,N}_0-X^{j,N}_0 \|^2 + 2 \int_0^{t\wedge \tau_n} (X^{i,N}_s - X^{j,N}_s) \cdot \dd (B^i_s-B^j_s)\notag \\
 &+ 4\Big(1-\frac{\theta}{N}\Big)(t\wedge \tau_n) - \frac{2\theta}{N}J_t^n,
\end{align}
where 
\begin{align*}
J_t^n =  \sum_{\substack{k=1 \\ k\ne i,j}}^N \int_0^{t\wedge \tau_n} \Big( \frac{X^{i,N}_s - X^{k,N}_s}{\|X^{i,N}_s - X^{k,N}_s\|^2} + \frac{X^{k,N}_s - X^{j,N}_s}{\|X^{k,N}_s - X^{j,N}_s\|^2} \Big) \cdot (X^{i,N}_s - X^{j,N}_s) \dd s.
\end{align*}
Using \eqref{xcv} together with the fact that $R_{\ig 1,N \id}(x) = (2N)^{-1} \sum_{1\le i\ne j\le N} \|x^i-x^j\|^2$ imply by summing the last equality over $i \ne j \in \ig 1,N \id$, and dividing  it  by $2N$, 
\begin{align*}
R_{\ig 1,N \id} (X^N_{t\wedge \tau_n}) =&R_{\ig 1,N \id} (X^N_0) +  2\int_0^{t\wedge \tau_n}\frac{1}{2N} \sum_{1\le i\ne j\le N} (X^{i,N}_s - X^{j,N}_s) \cdot \dd (B^i_s-B^j_s) \\
&+ (N-1)\Big(2-\frac{2\theta} N \Big)(t\wedge \tau_n) - \frac{\theta}{N^2} \int_0^{t\wedge \tau_n} S(X^N_s) \dd s,
\end{align*}
where for all  $x = (x^1,\dots, x^N)\in E_2 := \{ x\in (\rr^2)^N : \mbox{ for all } i,j\in \ig 1,N\id, x^i\ne x^j\}$ ,
$$
S(x) = \sum_{\substack{1\le i,j,k \le N \\ \textit{ distinct}}}  \Big( \frac{x^i - x^k}{\|x^i - x^k\|^2} + \frac{x^k - x^j}{\|x^k - x^j\|^2} \Big) \cdot (x^i - x^j).
$$
 This is well defined since a.s., $ \int_0^{\tau_n}\indiq_{X^N_s \notin E_2}\dd s = 0$ according to Theorem \ref{fj1}.  By oddness, we have for all $x\in  E_2$, 
\begin{align*}
S(x) =& 2 \sum_{\substack{1\le i,j,k \le N \\ \textit{ distinct}}} \frac{x^i - x^k}{\|x^i - x^k\|^2}  \cdot (x^i - x^j) \\
=& 2 \sum_{\substack{1\le i,j,k \le N \\ \textit{ distinct}}} \Big(1 + \frac{x^i - x^k}{\|x^i - x^k\|^2}  \cdot (x^k - x^j)\Big) = 2N(N-1)(N-2) -S(x),
\end{align*}

so that $S(x) = N(N-1)(N-2)$. Since $(N-1)(2-2\theta /N) - \theta (N-1)(N-2)/N = (N-1)(2-\theta)$, we deduce that we can write for all $n\ge 0$,
\begin{align}\label{qqqq}
R_{\ig 1,N \id} (X^N_{t\wedge \tau_n}) =&R_{\ig 1,N \id} (X^N_0) + 2 \int_0^{t\wedge \tau_n} R_{\ig 1,N \id}(X^N_s)^{1/2} \dd W^n_s + (N-1)(2-\theta)(t\wedge \tau_n),
\end{align}
 with
$$
(W^n_t)_{t\ge 0} = \Big(\frac1{2N}\sum_{1\le i\ne j \le N}\int_0^{t\wedge \tau_n}  \frac{X^{i,N}_s-X^{j,N}_s}{R_{\ig 1,N \id}(X^N_s)^{1/2}} \cdot \dd (B^i_s - B^j_s)\Big)_{t\ge 0},
$$
which is well defined because $R_{\ig 1,N \id}(X^N_s)>0$ for all $s \in [0, t\wedge \tau_n]$ since $\tau_n < \tau^N_N$ a.s., 

\vip

According to \eqref{qqqq}, the definition of a squared Bessel process (see  the book of  Revuz-Yor \cite[Chapter XI]{ry}) and since $d_{\theta,N}(N) = (N-1)(2-\theta)$, it only remains to show that for all $n\ge 0$, $(W^n_t)_{t\in [0,\tau_n)}$ is a $1$-dimensional Brownian motion restricted to $[0,\tau_n)$ which we now do. For all $n\ge 0$, $\langle W^n \rangle _t$ equals
\begin{align*}
& \frac{1}{4N^2}\!\! \sum_{\substack{1\le i\ne j\le N\\1\le k\ne \ell\le N }} \Big\langle   \int_0^{\cdot} \frac{(X^{i,N}_s-X^{j,N}_s)}{ R_{\ig 1,N\id}(X^N_s)^{1/2} }   \dd ( B^i_s-B^j_s) , \int_0^{\cdot} \frac{ (X^{k,N}_s-X^{\ell ,N}_s)}{R_{\ig 1,N\id}(X^N_s)^{1/2}} \cdot \dd ( B^k_s-B^\ell _s) \Big\rangle_{t\wedge \tau_n}
\end{align*}
so that
\begin{align*}
\langle W^n \rangle _t=&\frac{1}{4N^2} \sum_{\substack{1\le i\ne j\le N\\1\le k\ne \ell\le N }}( \delta_{i,k} - \delta_{i,\ell}- \delta_{j,k} + \delta_{j,\ell}  ) \int_0^{t\wedge \tau_n} \frac{(X^{i,N}_s-X^{j,N}_s) \cdot (X^{k,N}_s-X^{\ell ,N}_s)}{R_{\ig 1,N\id}(X^N_s)} \dd s.
\end{align*}
By distinguishing cases according to whether one or two of Kronecker's symbols in the sum equals $1$, we get
\begin{align*}
\langle W^n \rangle _t=& \frac{1}{N^2} \sum_{1\le i\ne j \le N} \int_0^{t\wedge \tau_n} \frac{\|X^{i,N}_s-X^{j,N}_s\|^2}{R_{\ig 1,N\id}(X^N_s)}  \dd s + \frac{1}{N^2}  \int_0^{t\wedge \tau_n} \frac{\tilde{S}(X^N_s)}{R_{\ig 1,N\id}(X^N_s)}\dd s,
\end{align*}
where 
\begin{align*}
\tilde{S}(x):=&  \sum_{\substack{1\le i,j,k \le N \\ \textit{ distinct}}} (x^i-x^j) \cdot (x^i-x^k).
\end{align*}

By oddness we get

\begin{align*}
\tilde S (x) =&  \sum_{\substack{1\le i,j,k \le N \\ \textit{  distinct }}}\big( \|x^i-x^j\|^2 + (x^i-x^j) \cdot (x^j-x^k)\big)\\
&= 2N(N-2) R_{\ig 1,N \id}(x) -\tilde{S}(x),
\end{align*}
so that $\tilde{S}(x) = N(N-2) R_{\ig 1,N \id}(x)$ which implies that 
\begin{align*}
\langle W^n \rangle _t = \frac{2}{N}( t\wedge \tau_n )+ \frac{N-2}{N} (t \wedge \tau_n) = t \wedge \tau_n .
\end{align*}
The conclusion easily follows thanks to the L\' evy characterization.
\end{proof}

In the following Lemma, we show that if particles indexed by $K\subset \ig 1,N \id$ are far enough away from the other, then $(R_K(X^N_t))_{t\ge 0}$ behaves like a squared Bessel process with dimension $d_{\theta,N}(|K|)$. Indeed, we will explain in which sense  we can neglect the interactions between particles indexed by $K$ and particles indexed by $K^c$.

\begin{lemma}\label{girsanov}
For all $T>0$, $\alpha>0$, $\theta  >0$, $N\ge 3$ and $K\subset \ig 1,N \id$ with $|K| \in \ig 2,N-1 \id$, if $\tau$ is a stopping time for a filtration $(\cF_t)_{t\ge 0}$ and $(X^{i,N}_t)_{t\in [0,\tau), i\in \ig 1,N \id}$ is a $KS(\theta, N)$-process on $[0,\tau)$, then the family $\Big((R_K(X^N_{\sigma^k_\alpha\wedge T+t}))_{t\in [0, \tilde\sigma^k_{\alpha,K}\wedge T -\sigma^k_{\alpha,K}\wedge T)}\Big)_{k\ge 1}$ is a family of squared Bessel process with dimension $d_{\theta ,N}(|K|)$ restricted to $[0,\tilde\sigma^k_{\alpha,K}\wedge T  - \sigma^k_{\alpha,K}\wedge T)$ driven by  independent Brownian 
motions  under $\QQ^{\alpha, T}_K := \cE ( L^{\alpha, T}_K)\cdot \PP$ where  $$
L^{\alpha,T}_K =  \frac{\theta}{N}\sum_{i\in K} \int_{0}^ {T} \indiq_{s\in \cup_{k\ge 0} [\sigma_{\alpha,K}^k,\tilde\sigma_{\alpha,K}^k)}\sum_{j\notin K} \frac{X^{i,N}_s-X^{j,N}_s}{\|X^{i,N}_s-X^{j,N}_s\|^2}\dd B^i_s,
$$
with $(\sigma_{\alpha,K}^k)_{k\ge 1}$ and $(\tilde\sigma_{\alpha,K}^k)_{k\ge 0}$ defined by induction by setting $\tilde\sigma _{\alpha,K}^0 = 0$ and for all $k\ge 1$, 
\begin{gather*}
\sigma_{\alpha,K}^k = \inf \{ t\ge \tilde\sigma_{\alpha,K}^{k-1} : \min_{i\in K, j\notin K} \|X^{i,N}_t - X^{j,N}_t\| \ge \alpha \} \\
 \tilde\sigma_{\alpha,K}^k = \inf \{ t\ge \sigma_{\alpha,K}^{k} : \min_{i\in K, j\notin K} \|X^{i,N}_t-X^{j,N}_t\| \le \alpha /2 \},
\end{gather*}
with the convention $\inf \emptyset = \tau$.
\end{lemma}

\begin{proof}
We fix $T>0$, $\alpha>0$, $N\ge 3$ and $K\subset \ig 1,N \id$ with $|K|\in \ig 1,N-1\id$. We compute 
\begin{align*}
\langle L^{\alpha,\cdot}_K\rangle _T = & \frac{\theta^2}{N^2} \sum_{i\in K} \int_0^T \indiq_{s\in \cup_{k\ge 0} [\sigma_{\alpha,K}^k,\tilde\sigma_{\alpha,K}^k)} \Big\| \sum_{j\notin K} \frac{X^{i,N}_s-X^{j,N}_s}{\|X^{i,N}_s-X^{j,N}_s\|^2}\Big\|^2 \dd s\\
\le & \frac{4\theta^2|K^c|^2}{\alpha^2N^2} \sum_{i\in K} \int_0^T \indiq_{s\in \cup_{k\ge 0} [\sigma_{\alpha,K}^k,\tilde\sigma_{\alpha,K}^k)} \dd s\\
\le & \frac{4\theta^2|K^c|^2|K|T}{\alpha^2N^2},
\end{align*}
so that the Novikov condition is  satisfied. Applying  Girsanov's theorem, we get that under $\QQ^{\alpha,T}_K$, there exists a $2N$-dimensional Brownian motion $(\tilde B ^i_t)_{t\ge 0, i\in \ig 1,N\id}$ such that for all $k\ge 1$, all $t\in [\sigma^k_{\alpha,K}\wedge T, \tilde\sigma^k_{\alpha,K}\wedge T)$, all $i\in \ig 1,N \id$, 

$$
X^{i,N}_t = X^{i,N}_{\sigma^k_{\alpha,K}\wedge T} + \tilde B^i_t- \tilde B^i_{\sigma^k_{\alpha,K}\wedge T} - \frac{\theta |K| /N}{|K|}\sum_{j\in K \setminus \{i\}} \int_{\sigma^k_{\alpha,K}\wedge T}^t \frac{X^{i,N}_s-X^{j,N}_s}{\|X^{i,N}_s-X^{j,N}_s\|^2}\dd s.
$$
 Thus $((X^{i,N}_{\sigma^k_{\alpha,K}\wedge T+t})_{t\in [0, \tilde\sigma^k_{\alpha,K}\wedge T -\sigma^k_{\alpha,K}\wedge T ), i\in K}, k\ge 0 )$ is a family of $KS(|K|\!,\! |K|\theta /N)$-processes  restricted to $[0, \tilde\sigma^k_{\alpha,K}\wedge T -\sigma^k_{\alpha,K}\wedge T)$ driven by independent Brownian motions.  We used the fact that the sequence of processes $((\tilde B_{\sigma^k_{\alpha,K}\wedge T + t} - \tilde B_{\sigma^k_{\alpha,K}\wedge T})_{t\in [0, \tilde \sigma^k_{\alpha,K} \wedge T - \sigma^k_{\alpha,K}\wedge T)}))_{k\ge 1}$ is a family of  independent  restricted $2N$-dimensional Brownian 
motions  under $\QQ^{\alpha,T}_K$.

\vip
 Lemma \ref{bessel} concludes the proof. 
\end{proof}

The following Proposition shows regularity in some sense of the process at $\tau_3^N$.

\begin{prop}\label{alternative}
If $\theta =2$, $N\ge 2$, and $((X^{i,N}_t)_{t\in [0,\tau_3^N)})_{i\in \ig 1,N\id}$ is a $KS(\theta,N)$-process on $[0,\tau_3^N)$, then for all $K\subset \ig 1,N \id$ such that $|K|\ge 2$, we have a.s. the alternative
$$
\lim_{t\to \tau^N_{3}} R_K(X^N_t) = 0 \quad \mbox{ or } \quad \liminf_{t\to \tau^N_3} R_{K}(X^N_t) >0
$$
\end{prop}

\begin{proof}
We proceed by reverse induction on $|K|$. First, if $|K|=N$, then the stochastic process $(R_{\ig 1,N\id}(X^N_t))_{t\in [0,\tau_3^N)}$ is a squared Bessel process with dimension $d_{\theta,N}(N)$ restricted to $[0,\tau_3^N)$ according to Lemma \ref{bessel}, so that $\lim_{t\to \tau_3^N}R_{\ig 1,N\id}(X^N_t)$ exists a.s.

\vip

We fix $k\in \ig 3,N\id$ and suppose that we proved the result for all subsets of $ \ig 1,N\id$ with cardinal $ k $, we fix $K\subset \ig 1,N\id$ such that $|K| = k-1$ and we prove that the result is true for $K$. If there exists $i \notin K$ such that $ R_{K\cup \{ i\}} (x) = 0$ then it is clear that $ R_K(x)=0$. This,  the  continuity of $R_K$ for all $K\subset \ig 1,N \id$ and the induction hypothesis gives us that it suffices to handle the study on the event $\{\liminf_{t\to \tau^N_3} \min_{i\notin K} R_{K\cup \{ i\}}(X^N_t) >0\}$. We set

$$
\cA = \Big\{\liminf_{t\to \tau^N_3} R_{K}(X^N_t)=0,\,\, \limsup_{t\to \tau^N_3} R_{K}(X^N_t)>0, \,\, \liminf_{t\to \tau^N_3} \min_{i\notin K} R_{K\cup \{ i\}}(X^N_t) >0 \Big\} 
$$

and assume by contradiction that $\PP ( \cA ) >0.$ Since $\tau^N_3<\infty$ a.s., there exists $T>0$ such that

$$
\PP(\cA \cap \{ \tau_3^N\le T\}) >0.
$$

{\it Step 1.} There exists $\e \in \QQ_+^*$, $\eta \in \QQ_+^*$ small enough and $c \in \QQ_+^*$, $\alpha \in \QQ_+^*$, $T>0$ such that with positive probability $ \liminf_{t\to \tau_3^N}\min_{i\notin K} R_{K\cup \{ i\}}(X^N_t)\ge 2c$, $R_K(X^N_t)$ crosses the interval $[\e, 2\e]$ an infinite amount of time during $[0,\tau_3^N)$, $\tau_3^N\le T$ and if $R_K(x) \le 2\e$ and $\min_{i\notin K} R_{K\cup \{ i \} } (x) \ge c$, then $\|x^i-x^j\|\ge \alpha$ for all $i\in K$, $j\notin K$.

\vip 

{\it Step 2.} We build in this step a sequence of i.i.d squared Bessel processes from the empirical dispersion $(R_K(X^N_t))_{t\ge 0}$. Take $T$, $\alpha$, $c$ and $\e$ defined in Step 1 and we define the sequences of stopping times $(\sigma_{c,\e}^k)_{k\ge 1}$ and $(\tilde\sigma_{c,\e}^k)_{\ge 0}$ by induction by setting $\tilde\sigma _{c,\e}^0 = 0$ and for all $k\ge 1$, 
\begin{gather*}
\sigma_{c,\e}^k = \inf \{ t\ge \tilde\sigma_{c,\e}^{k-1} : R_K(X^N_t) \le \e \quad\mbox{ and } \quad \min_{i\notin K} R_{K\cup\{i\}}( X^N_t) \ge 2c \} \\
 \tilde\sigma_{c,\e}^k = \inf \{ t\ge \sigma_{c,\e}^{k} : R_K(X^N_t) \ge 2\e \quad\mbox{ or }\quad \min_{i\notin K} R_{K\cup\{i\}}(X^N_t) \le c \},
\end{gather*}
with the convention $\inf \emptyset = \tau_3^N$. 

\vip

According to Lemma \ref{girsanov}, there exists a probability $\QQ^{\alpha,T}_K$ absolutely continuous with respect to $\PP$ such that $\Big((R_K(X^N_{\sigma^k_{c,\e}\wedge T +t}))_{t\in [0, \tilde\sigma^k_{c,\e}\wedge T -\sigma^k_{c,\e}\wedge T)}\Big)_{k\ge 1}$ is a family of squared Bessel processes with dimension $d_{\theta,N}(|K|)$ driven by  independent 
Brownian motions.  Indeed, one has to observe that a.s., for all $k\ge 1$, there exists $k_0 \ge 1$ such that $[\sigma^k_{c,\e}\wedge T, \tilde\sigma^k_{c ,\e}\wedge T] \subset [\sigma^{k_0}_{\alpha,K}\wedge T, \tilde\sigma^{k_0}_{\alpha,K }\wedge T]$, where $\sigma^{k_0}_{\alpha,K}$ and $\sigma^{k_0}_{\alpha,K}$ have been defined in Lemma \ref{girsanov}, because if $t\in [\sigma^k_{c,\e}\wedge T, \tilde\sigma^k_{c ,\e}\wedge T]$, then $R_K(X^N_t) \le 2\e$ and $\min_{i\notin K} R_{K\cup \{ i\}} (X^N_t) \ge c$, so that $\min _{i\in K, j\notin K} \|X^{i,N}_t - X^{j,N}_t\|^2 \ge \alpha$ according to Step 1.

\vip

 We consider a sequence $((W^{k}_t)_{t\ge 1})_{k\ge 0}$ of  independent  $1$-dimensional $\QQ^{\alpha,T}_K$-Brownian motions  independent  of all  the  other variables we have considered and we define for all $k\ge 1$,  
 \begin{align*}
 Z^k_t =& \indiq_{t\in [0, \tilde\sigma^k_{c,\e}\wedge T -\eta^{k,1})} R_K(X^N_{\eta^{k,1}+t})+ \indiq_{t\in  [\tilde\sigma^k_{c,\e}\wedge T -\eta^{k,1}, \tilde\sigma^k_{c,\e}\wedge T -\eta^{k,1} + \eta^{k,2}]} Z^{k,2}_t
 \end{align*} 
if $\sigma^k_{c,\e}\wedge T <\tau_3^N$ where 
\begin{itemize}
\item  $\eta^{k,1} = \inf \{ t\ge \sigma^k_{c,\e} : R_K(X^N_{t}) \ge \e \}\wedge T$,
\item $(Z^{k,2}_t)_{t\ge 0}$ is a $\QQ^{\alpha,T}_K$-squared Bessel process with dimension $d_{\theta,N}(|K|)$ driven by the Brownian motion $(W^{k,2}_t)_{t\ge 0}$ starting at $R_K(X^N_{\tilde\sigma^k_{c,\e}\wedge T })$, (observe that if $\tilde\sigma^k_{c,\e}\wedge T = \tau_3^N$, this has a sense since a squared Bessel process is continuous),
\item  $\eta^{k,2} = \inf \{ t\ge 0 : Z^{k,2}_t\ge 2\e \}$,
\end{itemize}
 
and $Z^k_t = \indiq_{t\in [0,\eta^{k,2}]}Z^{k,2}_t$ if $\sigma^k_{c,\e}\wedge T = \tau_3^N$.

\vip
 
 Under $\QQ^{\alpha,T}_K$, $((Z^k_t)_{t\ge 0})_{k\ge 0}$ is a family of i.i.d squared Bessel processes with dimension $d_{\theta,N}(|K|)$ starting at $\e$ and killed when reaching $2\e$. 

\vip

{\it Step 3.} We conclude. According to Step 2, $(\eta^{k,2})_{k\ge 0}$ is a i.i.d family of random variable under $\QQ^{\alpha,T}_K$ such that $\E_{\QQ^{\alpha,T}_K}( \eta^{k,2})>0 $, which together with the Borel-Cantelli Lemma imply
\begin{align}\label{xxxx}
\sum_{k\ge 1} \eta ^{k,2} = \infty \mbox{ a.s.}
\end{align}
However, we deduce from Step 1 that with positive probability, $(R_K(X^N_t))_{t\in [0, \tau_3^N)}$ up crosses $[\e,2\e]$ an infinite number of times and there exists $t_0 \in [0,\tau_3^N)$ such that for all $t\ge t_0$, $\min_{i\notin K} R_{K\cup \{ i\}}( X^N_t) \ge 2c$, so that for all $n\ge 0$, $\sigma^n_{c,\e}\wedge T<\tau_3^N$ and there exists $n_0\ge 1$ such that for all $n\ge n_0$,

$$
R_K( X^N_{\sigma^n_{c,\e}}) = \e \qquad \mbox{ and } \qquad R_K( X^N_{\tilde\sigma^n_{c,\e}}) = 2\e .
$$
Thus, with positive probability, there exists $n_0\ge 0$ such that for all $n\ge n_0$,
$$
(Z^{n}_t)_{t\in [0,\eta^{n,2})} = (R_K(X^N_t))_{t\in [ \sigma^n_{c,\e}, \tilde\sigma^n_{c,\e})},
$$ 
so that $\sum_{k\ge 1} \eta^{k,2} \le (\sum_{k=1}^{k_0-1} \eta^{k,2}) + \tau_3^N <\infty$. Indeed, $\tau^N_3 <\infty$ according to Theorem \ref{fj1} and $\sum_{k=1}^{k_0-1}\eta^{k,2}<\infty $ a.s. because $d_{\theta,N}(|K|) > 0$ (see  the book of  Revuz-Yor \cite[Chapter XI]{ry}), since $|K|<N$ and $\theta =2$. This is contradictory with \eqref{xxxx}.

\end{proof}

The following Lemma will be helpful to determine which kind of collision occurs. In particular, this gives us information on the dimension of the squared Bessel processes linked with the empirical variances.

\begin{lemma}\label{dimension}
If $N\ge 5$, then $k_2 \in \{ N-2,N-1 \}$.
\end{lemma}

\begin{proof}

We first have $d_{2,N}(N-1) = 2(N-2)/N <2$, so that $k_2 \le N-1$. Thus it suffices to show that $k_2 \ge N-2$. 

\vip

We can restrict our study to the case where $N\ge 6$ since the result is clear if $N=5$. It suffices to show that $2\le\min (d_{2,N}(3),d_{2,N}(N-3))$ since $d_{2,N}(k)$ is increasing on $(-\infty ,(N+1)/2]$ and decreasing on $[(N+1)/2, \infty )$. Indeed, $d_{2,N}(k) = 2(k-1)(N-k)/N$ is a polynomial expression with roots $1$ and $N$.

\vip

Moreover, since $d_{2,N}(3)\le d_{2,N}(N-3)$ because $|(N+1)/2-3| \ge |(N+1)/2-(N-3)|$, $d_{2,N}(k)$ is symmetric with respect to the line of equation $y=(N+1)/2$ and $d_{2,N}$ is increasing on $(-\infty, (N+1)/2)$, it suffices to show that $d_{2,N}(3)\ge 2$, but this is clear since $d_{2,N}(3) = 2(2-6/N) \ge 2$. 
\end{proof}

We are now ready to give the 

\begin{proof}[Proof of Proposition \ref{fj2}-(i)]
Let $(X^{i,N}_t)_{t\in [0,\tau_3^N), i\in \ig 1,N \id }$ be the $KS(\theta,N)$-process on $[0,\tau_3^N)$ built in Theorem \ref{fj1}. It suffices to show that $ \tau^N_3 = \tau^N_{k_2} $ a.s. We reason by the absurd and suppose that with positive probability, $\tau^N_{k_2} >\tau^N_3$, which implies that with positive probability there exists $\ell \ge 1$ such that $\tau^{N,\ell}_{k_2} >\tau_3^N$. Indeed, if for all $\ell \ge 1$, $\tau_{k_2}^{N,\ell}\le \tau^N_3$, then $\tau_{k_2}^N = \lim_{\ell \to \infty} \tau^{N,\ell}_3 \le \tau_3^N$. We conclude that under this assumption, there exists $T >0$ such that 
\begin{align}\label{hyp}
 \mbox{ with positive probability, }\quad \tau_3^N \le T \quad \mbox{ and there exists } \ell \ge 1 \mbox{ such that }\quad \tau^{N,\ell}_{k_2} \ge \tau^N_3.
 \end{align} 
Observe that on the event $\{ \tau^{N,\ell}_{k_2} \ge \tau_3^N\}$, for all $K\subset \ig 1,N \id$ with $|K| = k_2$, we have that $\min_{t\in [0,\tau_3^N)} R_K(X^N_t) \ge 1/\ell$.

\vip

{\it Step 1.}
 We show that there exists $\alpha >0$ and $K\subset \ig 1,N \id$ with $|K|\in \ig 3, k_2 -1 \id$ such that with positive probability,
$$
 \lim_{t\to \tau_3^N} R_K(X^N_t) = 0 \quad \mbox{ and } \quad \liminf_{t\to \tau_3^N} \min_{i\in K, j\notin K} \|X^{i,N}_t-X^{j,N}_t\| \ge \alpha.
$$
According to Proposition \ref{alternative}, the definition of $\tau_3^N$ and \eqref{hyp}, we have with positive probability,
$$
\cS := \Big\{ K\subset \ig 1,N \id : |K|\in  \ig 3,k_2-1\id \quad \mbox{ and } \quad \lim_{t\to \tau_3^N} R_K(X^N_t) = 0\Big\}\ne \emptyset,
$$
so that according to Proposition \ref{alternative}  and \eqref{hyp} again, with positive probability there exists $K \subset \ig 1,N \id $ with $|K|\in \ig3, k_2-1\id$ (which is for example a set of $\cS$ which is maximal for the inclusion)  and $c >0$ such that 
$$
\lim_{t\to \tau_3^N} R_K(X^N_t) = 0 \quad \mbox{ and }\quad \liminf_{t\to\tau_3^N} \min_{i\notin K} R_{K\cup \{ i\} } (X^N_t) \ge c.
$$ 
 We conclude observing that for any fixed $c'>0$, there exists $\e >0$ and $\alpha >0$ such that if $R_K(x) \le \e$ and $\min_{i\notin K} R_{K\cup \{i\}}(x) \ge c'$, then $\min_{i\in K, j\notin K} \|x^i-x^j\| \ge \alpha$. 

\vip

{\it Step 2.}
We conclude. We apply Lemma \ref{girsanov} with $\alpha $, $T$ and $K$ defined in Step 1, so that under $\QQ^{\alpha,T}_K$ with positive probability, a squared Bessel process with dimension $d_{\theta,N}(|K|) \ge 2$ (because $|K| \in \ig 3,k_2-1 \id$ and by definition of $k_2$) tends to $0$ in finite time, which is absurd.  

\end{proof}

The issue now is to prove Proposition \ref{fj2}-(ii). According to Proposition \ref{fj2}-(i), for all $N\ge 5$ we have $\tau_3^{N,\ell}-\tau_{k_2}^{N,\ell} \to 0$ in probability as $\ell \to \infty$, so that we can find a deterministic increasing sequence $(\ell_N)_{N\ge 5}$ such that 
\begin{align}\label{ttt}
\lim_{N\to \infty}\PP ( |\tau_3^{N,\ell_N}-\tau_{k_2}^{N,\ell_N}|\ge 1)=0.
\end{align} 

Until now we set for all $N\ge 5$, $\kappa_N := \tau_{k_2}^{N,\ell_N}$ and $\beta_N = \tau_3^{N,\ell_N}$. We need to prove first a technical lemma.

\begin{lemma}\label{compactfini}
 If $\sup_{N\ge 5} \E[\|X^{1,N}_0\|^6]<\infty$, then 
$$
\mbox{ for all } N\ge 5, \mbox{ all } t\ge 0 \qquad \E [\| X^{1,N}_{t\wedge \kappa_N}\|^6 ] \le \Big( \sup_{N\ge 5} \E[\|X^{1,N}_0\|^6]+18t\Big)e^{18t}
$$
\end{lemma}

Let us point out that the power $6$ in this result is arbitrary and does not play any specific role. 
 
\begin{proof}
We fix $M>0$ and set the stopping time $\sigma_M = \inf \{ t\ge 0 :  \sum_{i=1}^N\|X^{i,N}_{t\wedge \kappa_N}\|^6 \le M\}$ with the convention $\inf \emptyset = \kappa_N$. We apply Lemma \ref{ff}-(i) with $\varphi(x) = x^3$ and taking the expectation so that we get
\begin{align*}
\E[ \| X^{1,N}_{t\wedge \sigma_M}\|^6]   =&  \| X^{1,N}_0\|^6 + 18\E\Big[\int_0^{t\wedge \sigma_M}\|X^{1,N}_s\|^4\dd s \Big] \\
& +\frac{6\theta}{N} \E\Big[\sum_{j=2}^N \int_0^{t\wedge \sigma_M} K(X^{1,N}_s - X^{j,N}_s)\cdot  \|X^{1,N}_s\|^4 X^{1,N}_s\dd s\Big].
\end{align*}
Using exchangeability of the family $\{(X^{i,N}_{t\wedge \sigma_M})_{t\ge 0}, i\in \ig 1,N \id\}$ which is given by Theorem \ref{fj1} and the fact that $\sigma_M$ is invariant by permutation of the family $((X^{i,N}_t)_{t\in [0,\tau_3^N)})_{i\in \ig 1,N \id}$, we get
\begin{align*}
\E[ \| X^{1,N}_{t\wedge \sigma_M}\|^6 ]   =& \E[ \| X^{1,N}_0\|^6 ]+ 18\E\Big[\int_0^{t\wedge \sigma_M}\|X^{1,N}_s\|^4\dd s \Big] \\
& +\!\frac{6\theta}{N^2} \E\Big[ \sum_{1\le i\ne j\le N} \int_0^{t\wedge \sigma_M}\!\!\! K(X^{i,N}_s - X^{j,N}_s)\!\cdot\! \|X^{i,N}_s\|^4 X^{i,N}_s\dd s\Big]\\
\le & \E[\| X^{1,N}_0\|^6] + 18 \E\Big[\int_0^{t\wedge \sigma_M}\|X^{1,N}_s\|^4\dd s \Big],
\end{align*}
since by oddness of $K$ and convexity of $x\mapsto \|x\|^6$, and recalling the definition of $K$,
$$
\sum_{1\le i\ne j \le N} K(x^i-x^j)\cdot \|x^i\|^4x^i = \frac{1}{2} \sum_{1\le i\ne j \le N} K(x^i-x^j)\cdot [ \|x^i\|^4x^i-\|x^j\|^4x^j] \le 0.
$$
Since $u\le 1 + u^{3/2}$ for all $u\ge 0$ and $t\wedge \sigma_M \le t$, we get 
\begin{align*}
\E[ \| X^{1,N}_{t\wedge \sigma_M}\|^6 ]   =& \sup_{N\ge 5}\E[ \| X^{1,N}_0\|^6 ]+ 18t + 18 \int_0^t \E [ \|X^{1,N}_{s\wedge \sigma_M}\|^6 \dd s,
\end{align*}
which implies the result thanks to Gronwall's Lemma and Fatou's Lemma because $\sigma_M$ goes to $\kappa_N$ a.s. as $M\to \infty$.

\end{proof}

We can finally give the

\begin{proof}[Proof of Proposition \ref{fj2}-(ii)]
{\it Step 1.} We first show that it is sufficient to consider $F^N_0$ such that $\sup_{N\ge 5} \E [ \|X^{1,N}_0\|^6] < \infty$. Indeed, suppose that we have shown the result for these initial conditions and take some $F^N_0 \in \cP_{sym,1}^*$ and the associated random variable $(X^{i,N}_0)_{i\in \ig 1,N\id}$. We set the $\sigma ( X^{i,N}_0, N\ge 5, i\in \ig 1,N\id )$-measurable random variable $\Lambda = 1+\max_{j\in \ig 1,N\id}\|X^{j,N}_0\|$ and define
\begin{align*}
\mbox{ for all } i\in \ig 1,N\id, \quad \tilde{X}^{i,N}_0 = \frac{X^{i,N}_0}{\Lambda}.
\end{align*}
 We denote by $\tilde{F}^N_0$ the law of $(\tilde{X}^{1,N}_0,\cdots, \tilde{X}^{N,N}_0)$ and we clearly have   $\sup_{N\ge 5}\E(\|\tilde{X}^{1,N}_0\|^6) \le 1$ and $\tilde{F}^N_0 \in \cP_{sym ,1}^*$. We set $
\tilde{X}^{N}_t = X^{N}_{\Lambda^2t}/\Lambda$ and we see thanks to a change of variable and thanks to the scaling property of the Brownian motion that $(\tilde{X}^N_t)_{t\ge 0}$ is a $KS(2,N)$-process with initial law $\tilde{F}^N_0$. By hypothesis, we have that $\tilde{\tau}_{3}^{N,\ell_N} \to _{N\to \infty} \infty$ in probability with obvious notations. But since a.s. $\Lambda \ge 1$, we have a.s.
$$
\mbox{ for } N \mbox{ large enough, }\quad \beta_N=\tau_{3}^{N,\ell_N} \ge \Lambda^2\tilde\tau_{3}^{N,\lfloor \Lambda^2\ell_N \rfloor } \ge \tilde{\tau}_{3}^{N,\ell_N},
$$
which proves the result.

\vip

{\it Step 2.} We are reduced to show the result when $\sup_{N\ge 5} \E[\|X^{1,N}_0\|^6]<\infty$.

\vip

{\it Step 2.1.} We set $\rho_N := \inf \{ t\ge 0: R_{\ig 1,N\id} (X^N_t) \le N^{1/2} \}$ and show $\PP( \rho_N \le t) \to 0$ as $N\to \infty$ for all $t> 0$. Considering $(Y_t)_{t\in [0,\tau_3^N)} = (R_{\ig 1,N \id}(X^N_t)^{1/2})_{t\in [0,\tau_3^N)}$, Lemma \ref{bessel} and the It\^o's formula imply that for all $t\ge 0$,
\begin{align*}
Y_{t\wedge \rho_N} &= R_{\ig 1,N \id}(X^N_0)^{1/2}+ W_{t\wedge \rho_N} -\frac12 \int_0^{t\wedge \rho_N}\frac{\dd s}{Y_s} \\
&\ge R_{\ig 1,N \id}(X^N_0)^{1/2} + W_{t\wedge \rho_N} -\frac{N^{-1/4}}{2} (t\wedge \rho_N) ,
\end{align*}
where $(W_t)_{t\ge 0}$ is a $1$-dimensional Brownian motion. For $N$ large enough, we get
\begin{align*}
\PP ( \rho_N \le t )\le & \PP\Big(  R_{\ig 1,N \id}(X^N_0)^{1/2}+ \inf_{s\in [0,t]}( W_s - s/(2N^{1/4}) ) \le N^{1/4} \Big) \\
= &\PP\Big( \sup_{s\in [0,t]}( -W_s + s/(2 N^{1/4}) ) \ge R_{\ig 1,N \id}(X^N_0)^{1/2} - N^{1/4} \Big)\\
\le & \PP \Big( \sup_{s\in [0,t]} -W_s \ge R_{\ig 1,N \id}(X^N_0)^{1/2} - N^{1/4}-t/(2N^{1/4}) \Big)\\
\le & \PP \Big( \sup_{s\in [0,t]} -W_s > C_t N^{1/3} \Big) + \PP \Big( R_{\ig 1,N \id}(X^N_0)\le N^{2/3}\Big),
\end{align*}
where $C_t$ is a constant only depending on $t$. Since $(-W_t)_{t\ge 0}$ is a $1$-dimensional Brownian motion, it suffices to show that the second term of the RHS goes to $0$ as $N\to \infty$. Since $f_0$ is not a Dirac mass, we can fix $A>0$ such that
$$
\int_{\rr^2}\psi_A (x) f_0 (\dd x) >0,
$$
 where $\psi_A (y) = \|\varphi_A(y) - \int_{\rr^2} \varphi_A (x) f_0(\dd x) \|^2$ with $\varphi_A \in C_c(\rr^2,\rr^2)$ is a $1$-Lipschitz bounded function such that $\varphi_A(x) = x$ for every $x\in B(0,A)$. Since $\varphi_A$ is $1$-Lipschitz, we have
 \begin{align*}
 \sum_{i=1}^N \Big\|\varphi_A(x^i)-\frac1N \sum_{k=1}^N \varphi_A( x^k)\Big\|^2 =& \frac{1}{2N}\sum_{1\le i\ne j \le N} \|\varphi_A(x^i)-\varphi_A(x^j)\|^2 \\
 \le& \frac{1}{2N}\sum_{1\le i\ne j \le N} \|x^i-x^j\|^2 = R_{\ig 1,N \id}(x),
\end{align*}
so that
\begin{align*}
 \PP ( R_{\ig 1,N \id}(X^N_0)  \le  N^{2/3}) \le & \PP \Big( \sum_{i=1}^N \Big\| \varphi_A(X^{i,N}_0) - \frac{1}{N}\sum_{k=1}^N \varphi_A(X^{k,N}_0)\Big\|^2\le  N^{2/3}\Big)\\
 \le & P_1^N + P_2^N,
\end{align*}
where 
\begin{gather*}
P_1^N = \PP\Big(\Big\| \int_{\rr^2} \varphi_A(x) f_0(\dd x) - \int_{\rr^2} \varphi_A(x) \mu^N_0(\dd x) \Big\| \ge \e \Big),\\
P_2^N = \PP \Big( \frac2N\sum_{i=1}^N \Big\| \varphi_A(X^{i,N}_0) - \int_{\rr^2} \varphi_A(x) f_0(\dd x) \Big\|^2 -2\e^2 \le  N^{-1/3}\Big),
\end{gather*}
 because $\int_{\rr^2} \varphi_A(x) \mu^N_0(\dd x) = N^{-1} \sum_{k=1}^N \varphi_A(X^{k,N}_0)$, with $\e\! \in \!(\!0\!,(1/2)\!\int_{\rr^2} \!\psi_A (x) f_0 (\dd x)]^{1/2}] $.
By hypothesis on $\mu_0^N$, we get that $\lim_{N\to \infty} P_1^N =0$, whence it suffices to show that $\lim_{N\to\infty} P_2^N = 0$. It follows
\begin{align*}
P_2^N = & \PP \Big( 2\int_{\rr^2} \psi_A (y) \mu^N_0(\dd y) -2\e^2 \le N^{-1/3} \Big)\\
\le &  \PP \Big( \Big|\int_{\rr^2} \psi_A (y) \mu^N_0(\dd y) - \int_{\rr^2} \psi_A (y) f_0(\dd y) \Big| \ge \e^2 \Big)\\
& + \PP \Big( 2\int_{\rr^2} \psi_A (y) f_0(\dd y) -4\e^2 \le N^{-1/3} \Big),
\end{align*}
and these both quantity tends to $0$ as $N$ tends to $\infty$ because $2\int_{\rr^2} \psi (y) f_0(\dd y) -4\e^2 >0$, by definition of $\mu^N_0$ and since $\psi _A$ is continuous and bounded.

\vip

{\it Step 2.2.} We show that $\kappa_N$ goes to infinity in probability as $N\to \infty$. We fix $t>0$ and we have
\begin{align*}
\PP (\kappa_N \le t) &\le \PP( \rho_N \le t) + \PP( \rho_N> t, \kappa_N \le t).
\end{align*}
According to Step 2, it is sufficient to show that the last term of the previous inequality tends to $0$. On the event $\{\rho_N > \kappa_N\}$, there exists $K\subset \ig 1,N \id$ with $|K| = k_2$ such that $R_K(X^N_{\kappa_N}) \le 1/\ell_N$ and by definition of $\rho_N$ we have
\begin{align*}
 N^{1/2} &\le R_{\ig 1,N \id}(X^N_{\kappa_N})= \frac1{2N} \sum_{1\le i\ne j \le N} \|X^{i,N}_{\kappa_N}-X^{j,N}_{\kappa_N}\|^2 \\
&= \frac{|K|}{N}R_{K}(X^N_{\kappa_N}) + \frac{1}{N}\sum_{i\in K, j\notin K}\|X^{i,N}_{\kappa_N}-X^{j,N}_{\kappa_N}\|^2 + \frac{1}{2N}\sum_{i,j \notin K} \|X^{i,N}_{\kappa_N}-X^{j,N}_{\kappa_N}\|^2.
\end{align*}
Fixing $i_0 = \min K$, using that $\|x+y\|^2\le 2(\|x\|^2+\|y\|^2)$ in the two last sums and that $R_K(X^N_{\kappa_N}) \le 1/\ell_N$ since $\kappa_N>\rho_N$, we get 
\begin{align*}
N^{1/2} &\le \frac{|K|}{N\ell_N} + \frac{2|K^c|}{N}\sum_{i\in K}\|X^{i,N}_{\kappa_N}-X^{i_0,N}_{\kappa_N}\|^2+ \frac{2|K|+2}N\sum_{j\notin K}\|X^{i_0,N}_{\kappa_N}-X^{j,N}_{\kappa_N}\|^2\\
&\le \frac{|K|}{N\ell_N} + \frac{4|K||K^c|}{N\ell_N} + \frac{2|K|+2}N\sum_{j\notin K}\|X^{i_0,N}_{\kappa_N}-X^{j,N}_{\kappa_N}\|^2.
\end{align*}
Using again that $\|x+y\|^2 \le 2( \|x\|^2 + \|y\|^2)$, we deduce
\begin{align*}
N^{1/2}&\le \frac{|K|}{N\ell_N} + \frac{4|K||K^c|}{N\ell_N} + \frac{8(|K|+1)|K^c|}{N}\max_{i\in K^c \cup\{i_0\}}\|X^{i,N}_{\kappa_N} \|^2,
\end{align*}
which implies, 
\begin{align*}
N^{1/2} &\le  \frac{9}{\ell_N} + 16\max_{i\in K^c \cup\{i_0\}}\|X^{i,N}_{\kappa_N} \|^2,
\end{align*}
since $|K| = k_2 \in \ig N-2,N-1 \id$ according to Lemma \ref{dimension}. Since $\ell_N \ge N-4$ by construction, and by exchangeability, this gives that there exists $C>0$ such that for $N$ large enough

\begin{align*}
\PP( \rho_N > t, \kappa_N \le t) &\le \PP \Big( \bigcup_{i\in \ig 1,N \id} \{\|X^{i,N}_{t\wedge \kappa_N}\|^2 \ge C  N^{1/2}\} \Big) \\
&\le N \PP ( \|X^{1,N}_{t\wedge \kappa_N}\|^2 \ge C  N^{1/2}).
\end{align*}
Using the Markov inequality, this leads us to 
\begin{align*}
\PP( \rho_N > t, \kappa_N \le t)&\le \frac{\E [ \|X^{1,N}_{t\wedge \kappa_N}\|^6 ]}{C^3 N^{1/2}} ,
\end{align*}
which tends to $0$ as $N\to \infty$ thanks to Lemma \ref{compactfini}.

\vip

{\it Step 2.3.}
Here we conclude, but we clearly have $\beta_N \to \infty$ in probability as $N$ goes to infinity because recalling \eqref{ttt} and Step 2.2, for each $t\ge 0$ we get

$$
\PP(\beta_N\le t) \le \PP ( |\kappa_N -\beta_N|\le 1) + \PP ( \kappa_N \le t+1 ) \to_{N\to \infty} 0.
$$

\end{proof}

\section{The critical case}\label{critical}

Here we prove Theorem \ref{thm2}-(ii). We recall that $\theta =2$, the process 
$(X^{i,N}_t)_{t\in [0, \tau_3^N), i\in \ig 1,N \id}$ is a $KS(\theta,N)$-process issued from Theorem \ref{fj1} and 
$\beta_N$ has been defined in Proposition \ref{fj2}. We introduce
$G : (\rr^2)^3 \mapsto \rr$ defined as follows. For $x,y,z\in (\rr^2)^3$, we set $X=x-y$, $Y=y-z$ and $Z=z-x$
and put
$$
G(x,y,z) =  \Big(L(\|X\|^2)X + L(\|Y\|^2)Y +L(\|Z\|^2)Z\Big)\cdot
\Big( \frac{X}{\|X\|^2} + \frac{Y}{\|Y\|^2}+\frac{Z}{\|Z\|^2} \Big),
$$
if $0\notin \{X,Y,Z\}$, where $L(r) = \log(1+1/r) - 1/(1+r)$. If now $0\in \{ X,Y,Z\}$ and $ (X,Y,Z) \ne (0,0,0)$,
we put $G(x,y,z) = \infty$. If finally $X=Y=Z=0$, then $G(x,y,z) = 0$.
\vip
The following result and its proof are highly linked with Proposition \ref{estimgamma},
replacing $\varphi_a$ by $L(\cdot +a)$.  This estimate again indicates that the particles do not collide too much.

\begin{prop}\label{estimeeG}
If $f_0\in \cP(\rr^2)$ then for all $t>0$, 
\begin{gather*}
\sup_{N\ge 5}\E \Big[\int_0^{t\wedge \beta_N} G(X^{1,N}_s,X^{2,N}_s,X^{3,N}_s) \dd s \Big]<\infty, \label{bbb1}
\end{gather*}
\end{prop}

\begin{proof}
We fix $\eta \in (0,1]$ and we introduce $L_\eta(r)= L(\eta +r)$ for all $r\geq 0$.
We apply the It\^ o formula to $\varphi_\eta ( \|X^{1,N}_{t\wedge \beta_N} - X^{2,N}_{t\wedge \beta_N}\|^2)$ 
with $\varphi_\eta (r) = (r+\eta)\log ( 1+1/(\eta+r))$ whence $\varphi'_\eta(r) = L_\eta(r)$ 
and $\varphi_\eta''(r) = -[(\eta+r)(1+\eta +r)^2]^{-1}$. Taking the expectation, we find that
\begin{equation}\label{aaa1}
S^{N,\eta}_t:=\E[\varphi_\eta (\|X^{1,N}_{t\wedge \beta_N} - X^{2,N}_{t\wedge \beta_N}\|^2)]
= S^{N,\eta}_0 + \int_0^{t} \Big(A^{N,\eta}_s + \frac{2\theta}N B^{N,\eta}_s\Big)\dd s,
\end{equation}
where, setting $h_\eta(r)= -r/[(\eta+r)(1+\eta+r)^2]$,
\begin{align*}
A^{N,\eta}_s=&4 \E \Big[ \indiq_{s\le \beta_N} \Big(\Big(1-\frac{\theta}N \Big)L_\eta (\|X^{1,N}_s-X^{2,N}_s\|^2)+ h_\eta(\|X^{1,N}_s-X^{2,N}_s\|^2)\Big) \Big] \\
\geq &4\Big(1-\frac\theta N \Big)\E \Big[ \indiq_{s\le \beta_N} L_\eta (\|X^{1,N}_s-X^{2,N}_s\|^2)\Big]-4
\end{align*}
because $h_\eta$ is bounded by $1$, and where $B^{N,\eta}_s$ equals
\begin{align*}
 \sum_{j=3}^N \E\Big[\!\indiq_{s\le \beta_N}\!\Big(\!K(X^{1,N}_s\!\!-X^{j,N}_s)\! +\! K (X^{j,N}_s\!\!-X^{2,N}_s) )\!\Big)\!\!\cdot\!\! (X^{1,N}_s\!\!-X^{2,N}_s) 
L_\eta (\|X^{1,N}_s\!\!-X^{2,N}_s\|^2)\Big],
\end{align*}
so that by exchangeability, $B^{N,\eta}_s/(N-2)$ equals
\begin{align*}
 \E\Big[\indiq_{s\le \beta_N}\Big(K(X^{1,N}_s\!\!-X^{3,N}_s) + K (X^{3,N}_s\!\!-X^{2,N}_s) )\Big)\!\!\cdot \! (X^{1,N}_s\!\!-X^{2,N}_s)L_\eta (\|X^{1,N}_s\!\!-X^{2,N}_s\|^2)\Big].
\end{align*}
 Symmetrizing by exchangeability and since $\theta =2$, we get
\begin{align}\label{mmm}
A^{N,\eta}_s+\frac{2\theta} N B^{N,\eta}_s \geq \frac{4(N-2)}{3N} 
\E [ \indiq_{s\le \beta_N}F_{N,\eta}(X^{1,N}_s,X^{2,N}_s,X^{3,N}_s) ]  
-4,
\end{align}
where
\begin{align*}
F_{N,\eta}(x,y,z) =&L_\eta (\|x-y\|^2)[1+(x-y)\cdot( K(x-z) + K(z-y)) ] \\
&+ L_\eta (\|y-z\|^2)[1+(y-z)\cdot ( K(y-x) + K(x-z) )] \\
&+ L_\eta (\|z-x\|^2)[1+(z-x)\cdot (K(z-y) + K(y-x)) ].
\end{align*}
Setting $X= x-y$, $Y= y-z$, $Z=z-x$ and recalling that $ K(X) = -\frac{X}{\|X\|^2}\indiq_{X\ne 0}$ we find for $x$, $y$, $z$ distinct,
\begin{align*}
F_{N,\eta}(x,y,z) =& L_\eta (\|X\|^2)\Big[1+X\cdot \Big( \frac{Z}{\|Z\|^2} + \frac{Y}{\|Y\|^2} \Big)\Big] \\
&  L_\eta (\|Y\|^2)\Big[1+Y\cdot\Big( \frac{X}{\|X\|^2} + \frac{Z}{\|Z\|^2} \Big) \Big]\nonumber\\
&+  L_\eta (\|Z\|^2)\Big[1+Z\cdot \Big( \frac{Y}{\|Y\|^2} + \frac{X}{\|X\|^2} \Big)\Big]\nonumber\\
=&  L_\eta (\|X\|^2)X\cdot \Big( \frac{X}{\|X\|^2}+\frac{Y}{\|Y\|^2} + \frac{Z}{\|Z\|^2} \Big) 
\nonumber\\
&+  L_\eta (\|Y\|^2)Y\cdot \Big(\frac{X}{\|X\|^2}+\frac{Y}{\|Y\|^2} + \frac{Z}{\|Z\|^2} \Big) 
\nonumber\\
&+  L_\eta (\|Z\|^2)Z\cdot \Big(\frac{X}{\|X\|^2} + \frac{Y}{\|Y\|^2}+ \frac{Z}{\|Z\|^2} \Big)
\nonumber\\
=:&   G_{\eta}(x,y,z) \nonumber,
\end{align*}
where we have put 
\begin{gather*}
G_{\eta}(x,y,z)= \Big(L_\eta(\|X\|^2)X + L_\eta(\|Y\|^2)Y +L_\eta(\|Z\|^2)Z\Big)\cdot
\Big( \frac{X}{\|X\|^2} + \frac{Y}{\|Y\|^2}+\frac{Z}{\|Z\|^2} \Big).
\end{gather*}
Inserting this into \eqref{aaa1} and \eqref{mmm}, we find
\begin{align*}
S^{N,\eta}_t \geq& S^{N,\eta}_0 - 4 t + \frac{4(N-2)}{3N} \E\Big[ \int_0^{t\wedge \beta_N} G_{\eta}(X^{1,N}_s,X^{2,N}_s,X^{3,N}_s) \dd s \Big].
\end{align*}
 Observing that $ \varphi_\eta$ is bounded, we deduce  that $\sup_{N\ge 5,\eta\in (0,1]}S_t^{N,\eta} <\infty$. 
Since moreover $S_0^{N,\eta}\geq 0$, we conclude that 
\begin{equation}\label{ddd}
\sup_{N\ge 5,\eta\in(0,1]}
\E \Big[\int_0^{t\wedge \beta_N} G_{\eta}(X^{1,N}_s,X^{2,N}_s,X^{3,N}_s)\dd s \Big]<\infty.
\end{equation}
Recall that $G_{\eta}$ is nonnegative thanks to Lemma \ref{inegalitebarycentre} because $L_\eta$ is nonincreasing and positive. Finally,  using Lemma \ref{inegalitebarycentre} and that $r\mapsto L_\eta(r)-L_{\eta'}(r)$ is nonincreasing and nonnegative if $0<\eta<\eta'$, one easily shows that 
$\eta \to G_{\eta}(x,y,z)$ increases to $G(x,y,z)$ as $\eta$ decreases to $0$ for every $x$, $y$, $z\in \rr^2$ distinct. We thus
also deduce \eqref{bbb1} from \eqref{ddd} by monotone convergence because a.s. $X^{1,N}_s$, $X^{2,N}_s$ and $X^{3,N}_s$ are distinct for a.e $s\in [0, \beta_N]$ according to Theorem \ref{fj1}.
\end{proof}

 By Theorem \ref{thm2}-(i), we can find $(N_k)_{k\ge 0}$ and
$(\mu_t)_{t\ge 0}\in C([0,\infty),\cP(\rr^2))$ such that $\lim_k N_k=\infty$ and $ (\mu^{N_k,\beta_{N_k}}_t)_{t\ge 0}$ goes to $(\mu_t)_{t\ge 0 }$ in law as $k\to \infty$, where $C([0,\infty),\cP(\rr^2))$ is
endowed with the uniform convergence on compact time intervals, $\cP(\rr^2)$ being endowed with the
weak convergence topology. We obviously have $\mu_0=f_0$.

\vip

The following Lemma shows that the dispersion of $(\mu_t)_{t\ge 0}$ increases with time, which implies  that if $f_0$ is not equal to a Dirac mass, then for any $t\ge 0$, $\mu_t$ is not equal to a Dirac.

\begin{lemma}\label{monotone}
It holds that
$$
\mbox{a.s. for all }t\ge 0, \mbox{ all } x_0\in \rr^2, \quad  \int_{\rr^2} \|x-x_0\|^2 \mu_t(\dd x) \ge \int_{\rr^2} \|x-x_0\|^2f_0(\dd x).
$$
\end{lemma}

If the RHS is infinite, this result implies that the LHS is also infinite.

\begin{proof}
 By the Skorokhod representation, we can find a probabilistic space $(\tilde \Omega, \tilde \cF, \tilde\PP)$ and some random variables $((\tilde{\mu}^{N_k}_t)_{t\ge 0})_{k\ge 0}$ and $(\tilde \mu_t)_{t\ge 0}$ such that the law of $(\tilde\mu^{N_k}_t)_{t\ge 0}$ equals the law of $(\mu^{N_k,\beta_{N_k}}_t)_{t\ge 0}$ for all $k\ge 0$, the law of $(\tilde \mu_t)_{t\ge 0}$ equals the law of $(\mu_t)_{t\ge 0}$ and $(\tilde\mu^{N_k}_t)_{t\ge 0}$ converges a.s. to $(\tilde\mu_t)_{t\ge 0}$ as $k\to \infty$. Since $(\mu_t)_{t\ge 0}$ and $(\tilde\mu_t)_{t\ge 0}$ have the same law, it suffices to show  the result  with $(\tilde \mu_t)_{t\ge 0}$ instead of $(\mu_t)_{t\ge 0}$. To this end, we divide the proof into several steps. 

\vip

 {\it Step 1.} We show that for all $B\in \tilde\cF$, all $x_0\in \rr^2$ and all $A > 0$, we have
 \begin{align}\label{iii}
\tilde\E\Big[\indiq_B \int_{\rr^2} \varphi_A( \|x-x_0\|^2) &\tilde\mu_t (\dd x)\Big]\ge   
\tilde\PP(B)\int_{\rr^2} \varphi_A ( \|x-x_0\|^2) f_0 (\dd x)\nonumber\\
& -Ct\Big(\Big(\frac{\sup_{N\ge 5, s\in [0,t]} \frac1N\sum_{i=1}^N\E [ \psi ( \|X^i_s-x_0\|^2 )]}{\psi ((A/2-\|x_0\|^2)\vee 0) }\Big) 
\wedge \tilde \PP(B)\Big),
\end{align}
 where  $\varphi_A(r) = \chi (r/A) r$ with $\chi \in C^\infty _c (\rr_+)$ such that $\chi (r) \in [0,1]$ for all $r\ge 0$, $\chi(r) = 1$ for all $r\in [0,1]$, and $\chi (r) = 0$ for all $r\ge 2$ and where we recall that $\psi$ is a nondecreasing function such that $\lim_{r\to \infty} \psi (r) = \infty$ defined in Proposition \ref{ascolici}.

\vip
 
We introduce $\psi_A(x) =\varphi_A(\|x-x_0\|^2)$ and observe that $\psi_A$, $\nabla \psi_A$
and $\nabla^2 \psi_A$ are bounded (uniformly in  $A\ge 1$) and that
$\nabla \psi_A(x) = 2(x-x_0) \varphi_A'(\|x-x_0\|^2)$ and 
$\Delta \psi_A(x) = 4(\varphi_A'(\|x-x_0\|^2) + \|x-x_0\|^2\varphi_A''(\|x-x_0\|^2))$.
We now apply Lemma \eqref{ff}-(iii) with $\psi_A$,
we multiply by $\indiq_B$ with $B\in \tilde\cF$
and we take the expectation in order to get,
\begin{align}\label{ccc}
 &S^{A,k}_t= S^{A,k}_0 + \tilde \E [ \indiq_B M^{A,k}_t ] + 2 \tilde\E \Big[ \indiq_B \int_0^t ( I^{1,A,k}_s + I^{2,A,k}_s )\dd s \Big],
\end{align}
where 
\begin{gather*}
S^{A,k}_t = \tilde\E\Big[\indiq_B \int_{\rr^2} \varphi_A ( \|x-x_0\|^2) \tilde\mu ^{N_k}_t (\dd x)\Big],\\
M^{A,k}_t = \frac{1}{N_k}\sum_{i=1}^{N_k} \int_0^t\nabla \psi_A (X^{i,N_k}_s) \dd B^i _s,\\
I^{1,A,k}_s = 2 \int_{\rr^2} [\varphi_A'(\|x-x_0\|^2) + \|x-x_0\|^2 \varphi_A''(\|x-x_0\|^2)] \tilde\mu_s^{N_k}(\dd x),
\end{gather*}
and $I^{2,A,k}_s$ equals
$$
\theta \int_{\rr^2}\!\int_{\rr^2}\!\!\!  K(x-y)\!\cdot\![(x-x_0)\varphi_A'(\|x-x_0\|^2)\!-\!(y-x_0)\varphi_A'(\|y-x_0\|^2)] \tilde\mu_s^{N_k}(\dd x) \tilde\mu_s^{N_k} (\dd y).
$$

Since $\varphi_A(r) = r$ for all $r \in [0, A]$
and since $\varphi_A'(\|x-x_0\|^2) + \|x-x_0\|^2\varphi_A''(\|x-x_0\|^2)$ is bounded  by a universal constant ,
\begin{align}\label{wwww}
I^{1,A,k}_s \ge &  2 \int_{\|x-x_0\|^2 \le A} \tilde \mu^{N_k}_s(\dd x) - C\int_{\|x-x_0\|^2 \ge A} \tilde \mu_s^{N_k}(\dd x)\nonumber \\
= & 2 \tilde \mu_s^{N_k}(\{x:  \|x-x_0\|^2 \le A \}) - C\tilde \mu_s^{N_k}(\{x:  \|x-x_0\|^2 \ge A \}).
\end{align}
Using again that $\varphi_A(r) = r$ for all $r \in [0, A]$, that $K(z)=-z/||z||^2\indiq_{\{z\neq 0\}}$ 
and that the Jacobian of $x\to (x-x_0)\varphi_A'(\|x-x_0\|^2)=\frac12\nabla \psi_A(x)$ is uniformly bounded,
\begin{align*}
I^{2,A,k}_s \ge & -\theta \int_{\|x-x_0\|^2\le A}\int_{\|y-x_0\|^2\le A} \indiq_{\{x\neq y\}}
\tilde \mu_s^{N_k} (\dd x) \tilde \mu_s^{N_k}(\dd y) \\
&- \! C \!\int_{\rr^2}\!\int_{\rr^2} \!\indiq_{\{ \|x-x_0\|^2\ge A \} \cup \{ \|y-x_0\|^2\ge A \}}\tilde \mu_s^{N_k}(\dd x)\tilde \mu_s^{N_k}(\dd y)\\
\ge & -\theta \tilde \mu_s^{N_k}\otimes \tilde\mu^{N_k}_s( \{x: \|x-x_0\|^2 \le A \}\otimes
\{y : \|y-x_0\|^2 \le A \} ) \\
&- 2C \tilde\mu_s^{N_k}( \{x: \|x-x_0\|^2\ge A \} )
\end{align*}
where we used in the last inequality that $\indiq_{C\cup D} \le \indiq_C + \indiq _D$  for all sets $C$, $D$.

We easily conclude that
\begin{equation}\label{eee}
I^{2,A,k}_s \ge -\theta \tilde \mu_s^{N_k}( \{x: \|x-x_0\|^2 \le A \})
- 2 C \tilde \mu_s^{N_k}( \{x: \|x-x_0\|^2\ge A \} ).
\end{equation}
Putting \eqref{ccc}, \eqref{wwww} and \eqref{eee} together, recalling that $\theta =2$ and using the 
Cauchy-Schwarz inequality, we find 
\begin{align}\label{sss}
S^{A,k}_t \ge &  S^{A,k}_0 - \sqrt{\PP (B)} \sqrt{ \tilde \E [(M^{A,k}_t)^2]}- 3C \E \Big[ \indiq_B\int_0^t \tilde\mu^{N_k}_s(\{ \|x-x_0\|^2\ge A\})\dd s \Big].
\end{align}
Finally, observing that 
\begin{align*}
\tilde \mu_s^{N_k}( \{ \|x-x_0\|^2\ge A) \!&\le\! \tilde \mu_s^{N_k}\!\Big( \Big\{ \|x\|^2\ge \Big(\frac{A}{2} - \|x_0\|^2\Big)\vee 0 \Big\} \Big)\!\\
& = \!\tilde \mu_s^{N_k}\Big( \Big\{ \psi( \|x\|^2)\ge \psi \Big(\Big( \frac{A}{2} - \|x_0\|^2\Big)\vee 0 \Big)\Big\} \Big),
\end{align*}
since $\psi$ is non-decreasing, applying the Markov inequality, using exchangeability of the family $((X^{i,N}_{t\wedge \beta_N})_{t\ge 0}, i\in \ig 1,N \id )$ for all $N\ge 5$ and injecting in \eqref{sss}, we get
\begin{align}\label{aaaax}
S^{A,k}_t\ge & S^{A,k}_0 - \sqrt{\tilde\PP (B)} \sqrt{ \tilde \E [(M^{A,k}_t)^2]}\\
&-Ct\Big(\Big(\frac{\sup_{N\ge 5, s\in [0,t]} \frac1N \sum_{i=1}^N\E [ \psi ( \|X^{i,N}_{s\wedge \beta_{N}}\|^2 )] }{\psi ((A/2-\|x_0\|^2)\vee 0)}\Big) \wedge \tilde\PP(B)\Big).
\end{align}
We easily check that $M^{A,k}_t \to _{k\to \infty} 0$ in $L^2$.
Moreover, since $F : \nu \in C([0,\infty),\cP(\rr^2)) \mapsto  \int_{\rr^2} \varphi_A(\|x-x_0\|^2)\nu _t(\dd x)$ 
is continuous and bounded, it holds that
$$
\lim_k S^{A,k}_t = \tilde \E\Big[\indiq_B \int_0^t \varphi_A(||x-x_0||^2) \tilde \mu_t(\dd x)\Big].
$$
Hence letting $k\to \infty$ in \eqref{aaaax} and using that $\tilde \mu_0=f_0$, we find \eqref{iii}.

\vip

{\it Step 2.} We conclude. We let $A$ going to infinity in \eqref{iii}, the monotone convergence 
theorem gives us that for all $B\in \tilde{\cF}$, all $x_0\in \rr^2$, 
$$
\tilde\E\Big[\indiq_B \int_{\rr^2}  \|x-x_0\|^2 \tilde\mu_t (\dd x)\Big]\ge  \tilde\PP(B)
\int_{\rr^2} \|x-x_0\|^2 f_0 (\dd x).
$$
Choosing $B=B_{t,x_0} = \{ \int_{\rr^2} \|x-x_0\|^2 \tilde{\mu}_t(\dd x) < \int_{\rr^2} \|x-x_0\|^2f_0(\dd x) \}$,
we conclude that $\PP (B_{t,x_0}) = 0$ (else we would have a contradiction).
We conclude by continuity that a.s., for all $t\geq 0$, all $x_0\in \rr^d$,
$\int_{\rr^2} \|x-x_0\|^2 \tilde{\mu}_t(\dd x)  \ge  \int_{\rr^2} \|x-x_0\|^2f_0(\dd x)$.
\end{proof}

\begin{prop}\label{diffu}
If $\max_{x\in \rr^2}f_0 (\{x\}) < 1$ i.e. if $f_0$ is not a Dirac mass, then $a.s.$, for a.e. every $t>0$, $\mu_t$ is a diffuse measure, i.e.
$$
\E \Big[ \int_0^\infty \int_{\rr^2}\int_{\rr^2} \indiq_{\{x=y\}}\mu_t (\dd x) \mu_t(\dd y) \dd t \Big]=0.
$$
\end{prop}

\begin{proof}

{\it Step 1.} We first show that

$$
\E\Big[ \int_0^t \int_{\rr^2}\int_{\rr^2}\int_{\rr^2} G(x,y,z) \mu_s(\dd x) \mu_s(\dd y)\mu_s(\dd z) \dd s \Big] < \infty.
$$
According to Proposition \ref{estimeeG}, there exists $C>0$ such that for all $k\ge 0$, all $M>0$, all $\e>0$, 
\begin{align*}
\E \Big[\int_0^{t\wedge\beta_{N_k}} \chi_\e (x,y,z)(G( X^{1,N_k}_s, X^{2,N_k}_s, X^{3,N_k}_s)\wedge M) \dd s \Big]\le C,
\end{align*}
where $\chi_\e (x,y,z) $ is a smooth approximation of $\indiq_{D_\e}$ with 
$D_\e=\{ (x,y,z): \|X\|^2+\|Y\|^2+\|Z\|^2 \ge  \e \}$ such that $\chi_{2\e} \le \indiq_{D_\e} \le \chi_{\e}$. We recall that we have set $X=x-y$, $Y=y-z$, and $Z=z-x$. This implies by exchangeability
\begin{align*}
& \E \Big[\int_0^{t}\! \int_{\rr^2} \int_{\rr^2} \int_{\rr^2}\chi_\e (x,y,z) ( G(x,y,z)\wedge M )\mu^{N_k,\beta_{N_k}}_s(\dd x) \mu^{N_k,\beta_{N_k}}_s(\dd y)  \mu^{N_k,\beta_{N_k}}_s(\dd z)\dd s \Big] \\
 \le&  \E \Big[\!\int_0^{t\wedge \beta_{N_k}}\!\!\! \int_{\rr^2}\! \int_{\rr^2}\! \int_{\rr^2}\!\!\chi_\e (x,y,z) (G(x,y,z)\wedge M)\mu^{N_k,\beta_{N_k}}_s(\dd x) \mu^{N_k,\beta_{N_k}}_s(\dd y)  \mu^{N_k,\beta_{N_k}}_s(\dd z)\dd s \Big]
 \\
 &+ M \E [ t-t\wedge \beta_{N_k} ] \\
 \le& \frac{(N_k-1)(N_k-2)}{N_k^2}C + M\frac{3N_k-2}{N_k^2} +M \E [ t-t\wedge \beta_{N_k}].
\end{align*}
 The last inequality is established distinguishing if $x$, $y$, $z$ are distinct or not.  Since $(x,y,z) \mapsto  \chi_\e (x,y,z)(G(x,y,z) \wedge M)$ is continuous and bounded, 
sending $k$ to infinity gives
\begin{align*}
 \E \Big[\int_0^{t}\! \int_{\rr^2} \int_{\rr^2} \int_{\rr^2}\chi_\e (x,y,z) ( G(x,y,z)\wedge M )\mu_s(\dd x) \mu_s(\dd y)  \mu_s(\dd z)\dd s \Big] \le C ,
\end{align*}
indeed, $\E[ t-t\wedge \beta_{N_k}] \le t\PP( \beta_{N_k} \le t) \to_{k\to \infty} 0$ according to Proposition \ref{fj2}-(ii). We conclude using the monotone convergence theorem twice with $\e \to 0$ and $M\to \infty$. 

\vip

{\it Step 2.} We show that a.s., for almost every $s\ge0$, all $x\in \rr^2$, $\mu_s(\{x\}) \in \{0,1\}$. If it was not the case, it would exist with positive probability a bounded set $A \in \cB(\rr_+)$ with positive Lebesgue measure such that for all $s\in A$ there exists $a_s \in (0,1)$, $x_s \in \rr^2$ and a  non-trivial  measure $g_s$ such that $\mu_s = a_s \delta_{x_s} + g_s$ and $g_s(\{x_s\})=0$. On this event, we would have for all $s\in A$,
\begin{align*}
\int_{\rr^2} \int_{\rr^2} \int_{\rr^2}G(x,y,z)\mu_s(\dd x) \mu_s(\dd y)  \mu_s(\dd z) \ge  a_s^2\int_{\rr^2}G(x_s,x_s,z)  g_s(\dd z) = \infty,
\end{align*}
Indeed, recall that $G(x,x,z) = \infty$ if $z\ne x$. 
 We get a contradiction with Step 1 by integrating time and taking the expectation. 
 Notice that the measurability of $s\mapsto a_s$, $s\mapsto x_s$, and $s\mapsto g_s$ are not required to 
justify this computation.
 
\vip

{\it Step 3.} By Lemma \ref{monotone} we see that a.s., for all $t>0$, $\mu_t$ is not a (full) Dirac measure. 
Indeed, if with positive probability there exists $t>0$ such that $\mu_t$ is a Dirac measure, 
say $\delta_{x_0}$ with $x_0\in \rr^2$, then we have 
$0=\int_{\rr^2} \|x-x_0\|^2 \mu_t(\dd x) \ge \int_{\rr^2} \|x-x_0\|^2 f_0(\dd x)$ 
which implies $f_0 = \delta_{x_0}$ and this is forbidden
since $\max_{x\in \rr^2} f_0(\{x\})<1$ by assumption.

\vip

{\it Step 4.} Gathering Steps 2 and 3, we conclude that a.s., $\sup_{x\in \rr^2}\mu_t(\{x\})=0$ for a.e. $t\geq 0$,
whence the result. 
\end{proof}

We finally give the 

\begin{proof}[Proof of Theorem \ref{thm2}-(ii)]

Recall that $\theta =2$, $f_0 \in \cP(\rr^2)$ and that $(\mu^{N,\beta_N})_{N\ge 5}$ is the corresponding 
family of empirical processes. We know by Theorem \ref{thm2}-(i), that the family $\{ (\mu^{N,\beta_N}_t)_{t\ge 0}, N\ge 5\}$ is tight, so we can consider $(N_k)_{k\ge 0}$ and a random variable
$(\mu_t)_{t\ge 0}$ belonging to $C([0,\infty),\cP(\rr^2))$ such that $\lim_k N_k=\infty$ and
$(\mu^{N_k,\beta_{N_k}}_t)_{t\ge 0}$ goes to $(\mu_t)_{t\ge 0}$ in law as $k\to \infty$. Moreover we have $\mu_0=f_0$ since by hypothesis, $\mu^N_0$ goes weakly to $f_0$ in probability as $N\to \infty$. It suffices to prove that $(\mu_t)_{t\ge 0}$ satisfies \eqref{ggg}, which we have already seen in Proposition \ref{diffu},
and is a weak solution to \eqref{EDP}-\eqref{EDP2}, which we now do.

\vip

We apply Lemma \ref{ff}-(iii) to $\varphi \in C^2_b(\rr^2)$ and get 
$$
I^1_t(\mu^{N_k,\beta_{N_k}}) -I^2_t(\mu^{N_k,\beta_{N_k}})= M_k(t)+R_{k}(t),
$$
where 
\begin{align*}
M_k(t) = \frac1{N_k} \int_0^t \sum_{i=1}^{N_k} \nabla \varphi (X^{i,N_k}_s)\cdot\dd B^i_s,
\end{align*}
for all $\nu \in C([0,\infty), \cP(\rr^2))$, 
\begin{gather*}
I^1_t(\nu) = \int_{\rr^2} \varphi (x) \nu_t(\dd x)-\int_{\rr^2} \varphi (x) \nu_0(\dd x)-\frac12 \int_0^t \int_{\rr^2} \Delta \varphi (x) \nu_s(\dd x)\dd s,\\
I^2_t(\nu)=\frac\theta2\int_0^t \int_{\rr^2}\int_{\rr^2}  K(x-y)\cdot [\nabla\varphi (x)-\nabla \varphi(y)] 
\nu_s(\dd x) \nu_s (\dd y) \dd s,
\end{gather*}
and
\begin{align*}
R_{k}(t) =& -\frac12 \int_{t\land \beta_{N_k}}^t \int_{\rr^2} \Delta \varphi (x) \mu^{N_k,\beta_{N_k}}_s(\dd x)\dd s,
\\
&+\frac\theta2\int_{t\land \beta_{N_k}}^t \int_{\rr^2}\int_{\rr^2}  K(x-y)\cdot [\nabla\varphi (x)-\nabla \varphi(y)] 
\mu^{N_k,\beta_{N_k}}_s(\dd x) \mu^{N_k,\beta_{N_k}}_s (\dd y) \dd s.
\end{align*}
Using that $||K(x)||\leq ||x||^{-1}\indiq_{\{x\neq 0\}}$ and that $D^2\varphi$ is bounded, one easily checks that
$$
\E[|R_k(t)|]\leq C \E[t-t\land \beta_{N_k}]\leq C t \PP(\beta_{N_k} \leq t),
$$
which tends to $0$ as $k\to \infty$ by Proposition \ref{fj2}. 

\vip

Hence we conclude exactly as in Step 2 of the proof of Theorem \ref{thm1}-(ii) that 
a.s., for each $t\geq 0$, $I^1_t(\mu) = I^2_t(\mu)$, which precisely means that
$\mu$ is a weak solution to \eqref{EDP}-\eqref{EDP2}. Observe that 
we may also use here that a.s., $\mu_s\otimes\mu_s(D)=0$ for a.e. $s\geq 0$,
where $D=\{(x,y)\in (\rr^2)^2 : x=y\}$,  by Proposition \ref{diffu}.
\end{proof}

\begin{appendix}

\section{Appendix}
\setcounter{thm}{0}
\setcounter{equation}{0}

We prove here a de La Vallée Poussin type lemma.
\begin{lemma}\label{vallepoussin}
If $(X_n)_{n\ge 0}$ is a tight sequence of random variables with values in $\rr_+$, then there exists some nondecreasing $\psi \in C^2(\rr_+)$ such that $\lim_{r\to \infty} \psi (r) = \infty$, $\psi (0) = 1$, $\psi (2r) \le C \psi (r)$ and $(1+r)|\psi'(r)| + r |\psi ''(r) |\le C$ for some constant $C>0$ and 
\begin{align*}
\sup_{n\ge 0} \E[ \psi (X_n)] < \infty.
\end{align*}
\end{lemma}

\begin{proof}
Since $(X_n)_{n\ge 0}$ is tight, there exists a sequence of increasing positive numbers $(a_k)_{k\ge 0}$ such that 
$$
\sup_{n\ge 0} \PP (X_n \ge a_k) \le 2^{-k}.
$$
Moreover we can choose this sequence such that $a_0= 1$ and for all $k\ge 0$, $a_{k+1} \ge 2a_k$. 
We set, for all $k\ge 0$, $m_k = (a_{k+1}-a_k)^{-1}$. For all $r\ge 0$, we define 
$$
\psi (r) = 1 + \sum_{k\ge 0} \chi (m_k(r-a_k)),
$$
where $\chi$ is a smooth nondecreasing function such that $\chi(x) = 0$ for all $x\le 0$ and $\chi (x) = 1$ for all $x\ge 1$. Observe that for every $r\ge 0$, $\psi(r)<\infty$ since the sum in the definition has a finite number 
of non-zero terms. 
Clearly, $\psi$ is smooth, nondecreasing and $\psi (0) = 1$. Noting that $\psi(r)\geq n+1$ for all
$r \geq a_n$ and all $n\geq 0$ 
(because then $m_k(r-a_k)\geq 1$ for all $k \in \ig 0,n-1\id$ by definition of $m_k$, with the convention
that $\ig 0,-1\id=\emptyset$), we conclude that
$\lim_{r\to \infty} \psi (r) = \infty$. 
Moreover, for all $n\ge 0$,  since $\chi \le \indiq_{\rr_+}$,
\begin{align*}
\E [ \psi (X_n) ] \le 1+\sum_{k\ge 0} \E [ \indiq_{X_n \ge a_k}]\le 1 + \sum_{k\ge 0} 2^{-k}=3.
\end{align*}
It is finally sufficient to prove that 
\begin{itemize}
\item (a) $\sup_{r\ge 0}[(1+r)|\psi'(r)| + r|\psi''(r)|] <\infty$,
\item (b) $\sup_{r\ge 0}\psi(2r)/\psi(r) <\infty$.
\end{itemize} 

\vip

We first prove (a). Since the sum in the definition of $\psi$ has a finite number of non-zero terms,
we get
\begin{gather*}
\psi' (r) = \sum_{k\ge 0} m_k \chi '(m_k(r-a_k)),\qquad
\psi''(r) = \sum_{k\ge 0} m_k^2 \chi''(m_k(r-a_k)). 
\end{gather*}
Since for every $k \ge 0$, $r\mapsto \chi'(m_k(r-a_k))$ is supported by $[a_k, a_k +m_k^{-1}] = [a_k,a_{k+1}]$ 
(by definition of $m_k$), and since for every $r\ge 0$, there exists a unique $k_r\ge 0$ such that 
$a_{k_r} \le r <a_{k_r+1} = a_{k_r} + m_{k_r}^{-1}$, we get $\psi' (r) =  m_{k_r} \chi'(m_{k_r}(r-a_{k_r}))$ 
so that  $(1+r)|\psi'(r)|$ is less than
\begin{gather*}
(1+a_{k_r} + m_{k_r}^{-1}) m_{k_r} \|\chi'\|_\infty = (m_{k_r}+m_{k_r}a_{k_r} + 1)\|\chi'\|_\infty \le \sup_{k\ge 0} (m_k+m_ka_k + 1)\|\chi'\|_\infty.
\end{gather*}
However, for all $k\ge 0$, $m_k a_k = a_k/(a_{k+1}-a_k) \le 1/2$ since $a_{k+1}\ge 2 a_k$ by construction, 
and the sequence $(m_k)_{k\ge 0}$ is clearly bounded, so that $\sup_{r\ge 0} (1+r)|\psi'(r)|$ is finite. Finally, using the same kind of arguments, we find
$$
r\psi''(r) \le \sup_{k\ge 0} (a_k + m_k^{-1})m_k^2 \|\chi''\|_\infty 
=\sup_{k\ge 0} (a_km_k + 1)m_k \|\chi''\|_\infty <\infty,
$$
since as already seen, the sequences $(a_km_k)_{k\geq 0}$ and $(m_k)_{k\geq 0}$ are bounded.
\vip

We end up by proving (b). Fix $r\geq 0$ and consider $k_r$ such that 
$a_{k_r} \le r <a_{k_r+1}$. Then we have $\psi(r) \geq k_r+1$, as seen a few lines above.
Moreover, since $2r<2a_{k_r+1}\leq a_{k_r+2}$, 
we have $\psi(2r) \leq k_{r}+3$ (because $m_k(2r-a_k)\leq 0$ for all $k\geq k_{r}+2$). All in all,
$$
\frac{\psi(2r)}{\psi(r)} \le \frac{k_r+3}{k_r+1}\le 3,
$$
which concludes the proof. 
\end{proof}

\end{appendix}

\section{Acknowledgments}
We would like to warmly thank Nicolas Fournier for fruitful conversations and remark. Moreover we thank the anonymous referees, an Associate
Editor and the Editor for their constructive comments that improved the
quality of this paper.
\end{document}